\documentclass[preprint,11pt]{elsarticle}

\usepackage{amssymb}
\usepackage{amsmath}
\usepackage{amsfonts} 
\usepackage{amsbsy}
\usepackage{amscd}
\usepackage{amsthm}
\usepackage{geometry}                
\usepackage{graphicx}
\usepackage{amssymb}
\usepackage{epstopdf}
\usepackage{mathtools}
\usepackage{mathrsfs}
\usepackage[capitalize]{cleveref}
\usepackage{mathrsfs} 
\usepackage[framemethod=tikz]{mdframed} 
\usepackage{circuitikz}
\usetikzlibrary{arrows}
\usetikzlibrary{patterns}

\newcommand{\Hidden}[1]{}

\newcommand{\F}{\mathcal{F}}

\newtheorem{theorem}{Theorem}
\newtheorem{lemma}[theorem]{Lemma}

\newtheorem{corollary}[theorem]{Corollary}

\theoremstyle{definition}
\newtheorem{remark}[theorem]{Remark}
\newtheorem{definition}[theorem]{Definition}
\newtheorem{example}[theorem]{Example}
\theoremstyle{remark}

\begin{document}
\begin{frontmatter}
\title{Spanning 2-Forests and Resistance Distance in 2-Connected Graphs}

\author[1]{Wayne Barrett \fnref{fn1}}
\ead{wb@mathematics.byu.edu}

\author[1]{Emily J. Evans\fnref{fn1}}
\ead{ejevans@mathematics.byu.edu}
\author[2]{Amanda E. Francis\fnref{fn1}}
\ead{aefr@umich.edu}
\author[1]{Mark Kempton} 
\ead{mkempton@mathematics.byu.edu}
\author[1]{John Sinkovic}
\ead{sinkovic@mathematics.byu.edu}

\address[1]{Department of Mathematics, Brigham Young University, Provo, UT 84602, USA}
\address[2]{Mathematical Reviews, American Mathematical Society, Ann Arbor, MI 48103, USA}



\fntext[fn1]{Supported by the Defense Threat Reduction Agency -- Grant Number HDTRA1-15-1-0049.} 

\begin{abstract}{A spanning 2-forest separating vertices $u$ and $v$  of an undirected connected graph is  a spanning forest with 2 components such that $u$ and $v$ are in distinct components.  Aside from their combinatorial significance, spanning 2-forests have an important application to the calculation of resistance distance or effective resistance.  The resistance distance between vertices $u$ and $v$ in a graph representing an electrical circuit with unit resistance on each edge is the number of spanning 2-forests separating $u$ and $v$ divided by the number of spanning trees in the graph. There are also well-known matrix theoretic methods for calculating resistance distance, but the way in which the structure of the underlying graph determines resistance distance via these methods is not well understood.  

For any connected graph $G$ with a 2-separator separating vertices $u$ and $v$, we show that the number of spanning trees and spanning 2-forests separating $u$ and $v$ can be expressed in terms of these same quantities for the smaller separated graphs, which makes computation significantly more tractable. An important special case is the preservation of the number of spanning 2-forests if $u$ and $v$ are in the same smaller graph. In this paper we demonstrate that this method of calculating resistance distance is more suitable for certain structured families of graphs than the more standard methods. We apply our results to count the number of spanning 2-forests and calculate the resistance distance in a family of Sierpinski triangles and in the family of linear 2-trees with a single bend. 

}
\end{abstract}


\end{frontmatter}

\section{Introduction}

Resistance distance in graphs has played a prominent role not only in circuit theory and chemistry~\cite{bapatdvi,mgt,doylesnell,KleinRandic,oldbook}, but also in combinatorial matrix theory~\cite{Bapatbook,YangKlein} and spectral graph theory~\cite{bapatdvi,mgt,chenzhang,SpielSparse}.  Many of the methods for calculating resistance distance, e.g., those making use of the  Laplacian matrix  are $O(n^3)$ where $n$ is the number of vertices of the graph. Furthermore, the relationship between these resistance distances and the structure of the underlying graph is not well understood except in special cases.   An under-utilized method of calculating the resistance distance between two vertices $u$ and $v$ in a graph $G$ is by determining the number of spanning 2-forests separating $u$ and $v$ in $G$ and the number of spanning trees of $G$ (see Definition~\ref{def:2forests} and Theorem~\ref{thm:2forests/trees}). Thus, if the number of spanning trees is known, calculating the number of spanning 2-forests and resistance distance are equivalent problems.  This work presents new reduction formulas for determining these quantities for 2-connected graphs.   We apply these results to a new family of linear 2-trees generalizing the work of~\cite{bef}.  We begin with the following notation and  definitions.

Let $G$ be an undirected graph in which multiple edges are allowed but loops are not.  Let $V(G)$ denote the vertex set of $G$ and unless otherwise specified $V(G)=\{1,2,\dots,n\}$. Finally, let $T(G)$ denote the number of spanning trees of $G$.  

\begin{definition}\label{def:2forests}
Given any two vertices $u$ and $v$ of $G$, a spanning 2-forest separating $u$ and $v$ is a spanning forest with two components such that $u$ and $v$ are in distinct components. The number of such forests is denoted by  $\F_G(u,v)$.  In addition, we occasionally consider spanning 2-forests separating a vertex $u$ from a pair of vertices $v$ and $w$.  We denote the number of these by $\F_G(u,\{v,w\})$.    
\end{definition}

It follows from the matrix tree theorem~\cite[p. 5]{spectraofgraphs} that for any $j \in \{1,2, \ldots, n\}$, $T(G)=\det L_G(j)$ where $L_G$ is the combinatorial Laplacian matrix of $G$, and $L_G(j)$ is the matrix obtained from $L_G$ by deleting the jth row and column.
The following identity (see \cite{chaik} and Th. 4 of \cite{bapatdvi}) is a relative of the matrix tree theorem: 
\[\F_G(u,v)=\det L_G(u,v),\]
where $L_G(u,v)$ is the matrix obtained from $L_G$ by deleting rows $u,v$ and columns $u,v$.  

If $G$ has a cut-vertex $w$ and $G=G_1 \cup G_2$ with $G_1$ and $G_2$ connected and $V(G_1) \cap V(G_2) = \{w\}$, then it is evident that 
\begin{align}
    T(G)&=T(G_1)T(G_2)\label{eq:cutvtx1}\\
     \F_G(u,v)&=\F_{G_1}(u,v)T(G_2) {\rm \ for\ } u,v \in V(G_1)\label{eq:cutvtx2}\\
     \F_G(u,v)&=\F_{G_1}(u,w)T(G_2)+T(G_1)\F_{G_2}(w,v) {\rm \ for \ } u \in V(G_1) {\rm \ and \ } v \in V(G_2).\label{eq:cutvtx3}
    \end{align}

It is natural to ask if reduction formulae such as these can be found for graphs with no cut vertex, and the answer is in the affirmative if the graph has a cut-set of size 2.  For any graph $G$ with a 2-separator $\{i,j\}$ and associated decomposition $G = G_1 \cup G_2$, one can express $T(G)$ in terms of $T(G_k)$ and $\F_{G_k}(i,j), k=1,2$, and, furthermore, for any pair of vertices $u,v \in V(G)$ one can express $\F_G(u,v)$ in terms of $T(G_k), \F_{G_k}(x,y)$,  and $\F_{G/ij}(x,ij)$ for $k \in \{1,2\}$, $x \in \{u,v\}$ and $y \in \{i,j\}$. Here $G/ij$ denotes the graph obtained by identifying vertices $i$ and $j$ (see Definition~\ref{def:identify}). This is Theorem \ref{thm:formulation1} and is one of the main results of the next section.  
This reduction is particularly effective if the sizes of $G_1$ and $G_2$ are comparable, and if there are multiple 2-separators.  

As previously mentioned we also consider the important and closely related concept of resistance distance or effective resistance.  Consider $G$ as an electric circuit with unit  resistance on each edge, and suppose one unit of current flows into vertex $i$ and one unit of current flows out of vertex $j$.  Then the resistance distance $r_G(u,v)$ between vertices $u$ and $v$ is the ``effective" resistance between $u$ and $v$.  Alternatively, one can give a mathematical formulation
\begin{equation*}
r_G(i,j) = (\mathbf{e}_i - \mathbf{e}_j)^T L_G^\dagger (\mathbf{e}_i - \mathbf{e}_j),
\end{equation*}
where $\dagger$ denotes the Moore-Penrose inverse.  
The following theorem~\cite[Th. 4 and (5)]{bapatdvi} gives the relationship between the resistance distance between $u$ and $v$ and the number of spanning 2-forests separating $u$ and $v$.

\begin{theorem}\label{thm:2forests/trees} Given a graph $G$, the resistance distance between vertices $u$ and $v$ is given by $$r_G(u,v)=\dfrac{\F_G(u,v)}{T(G)}.$$
\end{theorem}

Returning to the case where $G$ has a cut-vertex $w$ as described on the previous page, we divide \eqref{eq:cutvtx2} by $T(G_1)T(G_2)$ and applying Theorem \ref{thm:2forests/trees} we obtain $r_G(u,v)=r_{G_1}(u,v)$, a much shorter proof than the one given of the same result, Theorem 2.5 (Cut Vertex Theorem) in \cite{bef}.  Dividing (3) by $T(G_1)T(G_2)$ we see that if $w$ is a cut vertex of $G$ and $u$ and $v$ lie in distinct components of $G-w$, then
\begin{equation}\label{eq:cutvtxres}
    r_G(u,v) = r_{G_1}(u,w)+r_{G_2}(v,w).
\end{equation}

Aside from (\ref{eq:cutvtxres}), there seem to be few applications of Theorem \ref{thm:2forests/trees} to the calculation of resistance distance.  One significant example is the proof of the second statement of Theorem 7 in \cite{BapatWheels}. Our reduction formulae open the possibility of finding closed forms for resistance distances in many additional graphs.  We illustrate this for the Sierpinski triangle and the family of linear 2-trees with a single bend in Section 3  (see Figures \ref{fig:Sn_building} and \ref{fig:bent}). 

\begin{definition} A linear 2-tree (or 2-path) on $n$ vertices is a graph $G$ satisfying the following 4 properties.
\begin{itemize}
    \item $G$ has $2n-3$ edges.
    \item $K_4$ is not a subgraph of $G$.
    \item $G$ is chordal (every induced cycle is a triangle).
    \item $G$ has two degree two vertices.
    \end{itemize}
\end{definition}

Alternatively, a linear 2-tree is a graph $G$ that is constructed inductively by starting with a triangle and connecting each new vertex to the vertices of an existing edge that includes a vertex of degree 2.

\begin{definition}[straight linear 2-tree]\label{def:lin2treest}
A straight linear 2-tree is a graph $G_n$ with $n$ vertices with adjacency matrix that is symmetric, banded, with the first and second subdiagonals equal to one, the first and second superdiagonals equal to one, and all other entries equal to zero. See Figure~\ref{fig:2tree}. \end{definition} 
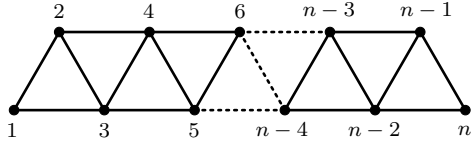
\begin{figure}[!ht]
\begin{center}

\begin{tikzpicture}[line cap=round,line join=round,>=triangle 45,x=1.0cm,y=1.0cm,scale = 1.2]
\draw [line width=1.pt] (-3.,0.)-- (-2.,0.);
\draw [line width=1.pt] (-2.,0.)-- (-1.,0.);
\draw [line width=1.pt,dotted] (-1.,0.)-- (0.,0.);
\draw [line width=1.pt] (0.,0.)-- (1.,0.);
\draw [line width=1.pt] (1.,0.)-- (2.,0.);
\draw [line width=1.pt] (2.,0.)-- (1.5,0.866025403784435);
\draw [line width=1.pt] (1.5,0.866025403784435)-- (1.,0.);
\draw [line width=1.pt] (1.,0.)-- (0.5,0.8660254037844366);
\draw [line width=1.pt] (0.5,0.8660254037844366)-- (0.,0.);
\draw [line width=1.pt,dotted] (0.,0.)-- (-0.5,0.8660254037844378);
\draw [line width=1.pt] (-0.5,0.8660254037844378)-- (-1.,0.);
\draw [line width=1.pt] (-1.,0.)-- (-1.5,0.8660254037844385);
\draw [line width=1.pt] (-1.5,0.8660254037844385)-- (-2.,0.);
\draw [line width=1.pt] (-2.,0.)-- (-2.5,0.8660254037844388);
\draw [line width=1.pt] (-2.5,0.8660254037844388)-- (-3.,0.);
\draw [line width=1.pt] (-2.5,0.8660254037844388)-- (-1.5,0.8660254037844385);
\draw [line width=1.pt] (-1.5,0.8660254037844385)-- (-0.5,0.8660254037844378);
\draw [line width=1.pt,dotted] (-0.5,0.8660254037844378)-- (0.5,0.8660254037844366);
\draw [line width=1.pt] (0.5,0.8660254037844366)-- (1.5,0.866025403784435);
\begin{scriptsize}
\draw [fill=black] (-3.,0.) circle (1.5pt);
\draw[color=black] (-3.02279181666165,-0.22431183338253265) node {$1$};
\draw [fill=black] (-2.,0.) circle (1.5pt);
\draw[color=black] (-2.0001954862580344,-0.22395896857501957) node {$3$};
\draw [fill=black] (-2.5,0.8660254037844388) circle (1.5pt);
\draw[color=black] (-2.5018465162673555,1.100526386601008717) node {$2$};
\draw [fill=black] (-1.5,0.8660254037844385) circle (1.5pt);
\draw[color=black] (-1.5081915914412003,1.100526386601008717) node {$4$};
\draw [fill=black] (-1.,0.) circle (1.5pt);
\draw[color=black] (-1.0065405614318794,-0.22290037415248035) node {$5$};
\draw [fill=black] (-0.5,0.8660254037844378) circle (1.5pt);
\draw[color=black] (-0.4952423962300715,1.100526386601008717) node {$6$};
\draw [fill=black] (0.,0.) circle (1.5pt);
\draw[color=black] (-0.03217990699069834,-0.22431183338253265) node {$n-4$};
\draw [fill=black] (0.5,0.8660254037844366) circle (1.5pt);
\draw[color=black] (0.4887653934035965,1.103344389715898) node {$n-3$};
\draw [fill=black] (1.,0.) circle (1.5pt);
\draw[color=black] (0.9904164234129174,-0.22431183338253265) node {$n-2$};
\draw [fill=black] (1.5,0.866025403784435) circle (1.5pt);
\draw[color=black] (1.5692445349621338,1.1042991524908385) node {$n-1$};
\draw [fill=black] (2.,0.) circle (1.5pt);
\draw[color=black] (1.993718483431559,-0.22431183338253265) node {$n$};
\end{scriptsize}
\end{tikzpicture}
%
%
\end{center}
\caption{A straight linear 2-tree}
\label{fig:2tree}
\end{figure}

In~\cite{bef} the authors obtained an explicit formula for the resistance distance between any two vertices in a straight linear 2-tree on $n$ vertices, and verified that the number of spanning trees of a straight linear 2-tree is $F_{2n-2}$, where $F_k$ is the kth Fibonacci number. Consequently, the number of spanning 2-forests separating two vertices can be found immediately from Theorem \ref{thm:2forests/trees}.  

A linear 2-tree with a single bend can be obtained from two straight linear 2-trees.  This fact and Theorem \ref{thm:formulation1} are applied in Section 3 to obtain an explicit formula for all resistance distances (all separating 2-forests) in the family of linear 2-trees with a single bend. If $u$ and $v$ are the end vertices of this ``bent" linear 2-tree on $n$ vertices, the number of spanning 2-forests separating $u$ and $v$ is less than the the number in the straight linear 2-tree by the product of four Fibonacci numbers. (See Corollary \ref{cor:27})


\section{2-Separations}

\begin{definition} A \emph{2-separation} of a graph $G$ is a pair of subgraphs $G_1,G_2$ such that 
\begin{itemize}
    \item $V(G) = V(G_1)\cup V(G_2)$,
    \item $|V(G_1)\cap V(G_2)|=2$, 
    \item $E(G)=E(G_1)\cup E(G_2)$, and 
    \item $E(G_1)\cap E(G_2)=\emptyset$.  
\end{itemize}
The pair of vertices, $V(G_1)\cap V(G_2)$, is called a \emph{2-separator} of $G$.

\end{definition}
    
Throughout this section, we will let $G$ denote a graph with a 2-separation, $G_1,G_2$ will denote the two graphs of the separation, and we will let $\{i,j\} = V(G_1)\cap V(G_2)$. Note that the graph resulting from the deletion of vertices $i$ and $j$ from $G$ is a disconnected graph. 

\begin{theorem}\label{2sep-tree}
Let $G$ be a graph with a 2-separation as above.  Then
\[
T(G) = T(G_1)\F_{G_2}(i,j) + T(G_2)\F_{G_1}(i,j).
\]
\end{theorem}
\begin{proof}
If $T$ is a spanning tree of $G$, then since $\{i,j\}$ separates $G_1$ and $G_2$, the unique path in $T$ connecting $i$ to $j$ must lie either entirely in $G_1$ or entirely in $G_2$.  Then the restriction of $T$ to the other side is a spanning 2-forest of that side separating $i$ and $j$.  By the multiplication principle, the result follows.  
\end{proof}

Before stating the first analogous result for 2-forests, we need the following definition.
\begin{definition}\label{def:identify}
If $i,j$ are vertices of $G$, then $G/ij$ is the graph obtained from $G$ by identifying vertices $i$ and $j$ into a vertex we denote by $ij$.  In the identification, any edge $\{u,i\}$ or $\{u,j\}$ with $u \ne i,j$ is replaced by an edge $\{u,ij\}$. (So if $i$ and $j$ have a common neighbor $v$, then there is a double edge from $v$ to $ij$ in the new graph.)  Any edge $\{u,v\}$ with neither $u$ nor $v$ equal to $i$ or $j$ remains.  If $\{i,j\}$ is an edge it disappears. 
\end{definition}

\begin{theorem}\label{2sep-forest-same-side}
Let $G$ be as above, and let $u,v\in V(G_1)$.  Then
\[
\F_{G}(u,v) = \F_{G_1}(u,v)\F_{G_2}(i,j) + \F_{G_1/ij}(u,v)T(G_2).
\]
\end{theorem}
\begin{proof}
Let $F$ be a spanning 2-forest in $G$ that separates $u$ and $v$.   

\textbf{Case 1:} $i$ and $j$ belong to different components of $F$ restricted to $G_2$.  Then the restriction of $F$ to $G_1$ has 2 components that separate $u$ and $v$ and the restriction to $G_2$ has two components that separate $i$ and $j$.  By the multiplication principle, the number of ways to do this is $\F_{G_1}(u,v)\F_{G_2}(i,j)$.

\textbf{Case 2:} $i$ and $j$ belong to the same component of $F$ restricted to $G_2$.  Then $i$ and $j$ are in the same component of $F$ and thus in the same component of $F$ as either $u$ or $v$.  Without loss of generality, suppose they are in the same component of $F$ as $u$.  

Since $i$ and $j$ belong to the same component of $F$ restricted to $G_2$, and $i,j$ is a 2-separator for $G$, $F$ restricted to $G_2$ is a spanning tree of $G_2$.  Then the path in $F$ from $i$ to $j$ is in $G_2$ and not in $G_1$, so the restriction of $F$ to $G_1$ has 3 components, separating $i$, $j$, and $v$ (and $u$ will be in a component with either $i$ or $j$).  Then identifying $i$ with $j$, we obtain a spanning 2-forest of $G_1/ij$ that separates $u$ from $v$.  There are $\F_{G_1/ij}(u,v)T(G_2)$ ways of doing this.
\end{proof}

\begin{definition}
If $G$ is a graph with a 2-separator $\{i,j\}$, and corresponding 2-separation $G_1,G_2$, then a \emph{2-switch} is the operation that identifies the copy of $i$ in $G_1$ with the copy of $j$ in $G_2$ and the copy of $j$ in $G_1$ with the copy of $i$ in $G_2$ (see Figure \ref{fig:2_switch}).
\end{definition}
We note that we are following the terminology adopted in \cite{2switch} and remark that the term 2-switch has also been used in a different context with a distinct meaning.

\begin{figure}[h]
    \centering
    \begin{tikzpicture}
    \begin{scope}[scale = .85]
\draw [line width=.8pt] (7.5,3.1339745962155607)-- (8.,4.);
\draw [line width=.8pt] (8.,4.)-- (7.,4.);
\draw [line width=.8pt] (7.,4.)-- (7.5,3.1339745962155607);
\draw [line width=.8pt] (7.5,3.1339745962155607)-- (6.5,3.133974596215561);
\draw [line width=.8pt] (6.5,3.133974596215561)-- (7.,4.);
\draw [line width=.8pt] (7.,4.)-- (6.,4.);
\draw [line width=.8pt] (6.,4.)-- (6.5,3.133974596215561);
\draw [line width=.8pt] (6.5,3.133974596215561)-- (5.5,3.133974596215561);
\draw [line width=.8pt] (5.5,3.133974596215561)-- (6.,4.);
\draw [line width=.8pt] (6.,4.)-- (5.,4.);
\draw [line width=.8pt] (5.,4.)-- (5.5,3.133974596215561);
\begin{scriptsize}
\draw [fill=black] (8.,4.) circle (1.6pt);
\draw [fill=black] (7.,4.) circle (1.6pt);
\draw [fill=black] (7.5,3.1339745962155607) circle (1.6pt);
\draw [fill=black] (6.5,3.133974596215561) circle (1.6pt);
\draw[color=black] (6.565892549850077,2.9389718782765817) node {$j$};
\draw [fill=black] (6.,4.) circle (1.6pt);
\draw[color=black] (5.855297516616688,4.257011052822387) node {$i$};
\draw [fill=black] (5.5,3.133974596215561) circle (1.6pt);
\draw [fill=black] (5.,4.) circle (1.6pt);
\end{scriptsize}
\node at (8.7,3.66) {$\to$};

    \end{scope}
    \begin{scope}[xshift = .3 \textwidth, scale = .85]
\draw [line width=.8pt] (7.5,3.1339745962155607)-- (8.,4.);
\draw [line width=.8pt] (8.,4.)-- (7.,4.);
\draw [line width=.8pt] (7.,4.)-- (7.5,3.1339745962155607);
\draw [line width=.8pt] (7.5,3.1339745962155607)-- (6.5,3.133974596215561);
\draw [line width=.8pt] (6.5,3.133974596215561)-- (7.,4.);
\draw [line width=.8pt] (7.,4.)-- (6.,4.);
\draw [line width=.8pt] (4.74564686057547,3.126877069979172)-- (5.752146624695713,3.1280818467983713);
\draw [line width=.8pt] (5.752146624695713,3.1280818467983713)-- (5.247853375304275,3.9991338230199465);
\draw [line width=.8pt] (5.247853375304275,3.9991338230199465)-- (4.74564686057547,3.126877069979172);
\draw [line width=.8pt] (4.74564686057547,3.126877069979172)-- (4.241353611184032,3.9979290462007486);
\draw [line width=.8pt] (4.241353611184032,3.9979290462007486)-- (5.247853375304275,3.9991338230199465);
\begin{scriptsize}
\draw [fill=black] (8.,4.) circle (1.6pt);
\draw [fill=black] (7.,4.) circle (1.6pt);
\draw [fill=black] (7.5,3.1339745962155607) circle (1.6pt);
\draw [fill=black] (6.5,3.133974596215561) circle (1.6pt);
\draw[color=black] (6.565892549850077,2.9389718782765817) node {$j$};
\draw [fill=black] (6.,4.) circle (1.6pt);
\draw[color=black] (5.855297516616688,4.257011052822387) node {$i$};
\draw [fill=black] (5.752146624695713,3.1280818467983713) circle (1.6pt);
\draw[color=black] (5.820913885976363,2.973355508916907) node {$j$};
\draw [fill=black] (5.247853375304275,3.9991338230199465) circle (1.6pt);
\draw[color=black] (5.144702483383298,4.257011052822387) node {$i$};
\draw [fill=black] (4.74564686057547,3.126877069979172) circle (1.6pt);
\draw [fill=black] (6.5,3.133974596215561) circle (1.6pt);
\draw [fill=black] (4.241353611184032,3.9979290462007486) circle (1.6pt);
\end{scriptsize}
\node at (8.7,3.66) {$\to$};

    \end{scope}
    \begin{scope}[xshift = .57\textwidth, scale = .85]
\draw [line width=.8pt] (7.5,3.1339745962155607)-- (8.,4.);
\draw [line width=.8pt] (8.,4.)-- (7.,4.);
\draw [line width=.8pt] (7.,4.)-- (7.5,3.1339745962155607);
\draw [line width=.8pt] (7.5,3.1339745962155607)-- (6.5,3.133974596215561);
\draw [line width=.8pt] (6.5,3.133974596215561)-- (7.,4.);
\draw [line width=.8pt] (7.,4.)-- (6.,4.);
\draw [line width=.8pt] (4.74564686057547,3.126877069979172)-- (5.752146624695713,3.1280818467983713);
\draw [line width=.8pt] (5.752146624695713,3.1280818467983713)-- (5.247853375304275,3.9991338230199465);
\draw [line width=.8pt] (5.247853375304275,3.9991338230199465)-- (4.74564686057547,3.126877069979172);
\draw [line width=.8pt] (4.74564686057547,3.126877069979172)-- (5.249940109966908,2.255825093757596);
\draw [line width=.8pt] (5.249940109966908,2.255825093757596)-- (5.752146624695713,3.1280818467983713);
\begin{scriptsize}
\draw [fill=black] (8.,4.) circle (1.6pt);
\draw [fill=black] (7.,4.) circle (1.6pt);
\draw [fill=black] (7.5,3.1339745962155607) circle (1.6pt);
\draw [fill=black] (6.5,3.133974596215561) circle (1.6pt);
\draw[color=black] (6.565892549850077,2.9389718782765817) node {$j$};
\draw [fill=black] (6.,4.) circle (1.6pt);
\draw[color=black] (5.855297516616688,4.257011052822387) node {$i$};
\draw [fill=black] (5.752146624695713,3.1280818467983713) circle (1.6pt);
\draw[color=black] (5.820913885976363,2.973355508916907) node {$i$};
\draw [fill=black] (5.247853375304275,3.9991338230199465) circle (1.6pt);
\draw[color=black] (5.144702483383298,4.257011052822387) node {$j$};
\draw [fill=black] (4.74564686057547,3.126877069979172) circle (1.6pt);
\draw [fill=black] (6.5,3.133974596215561) circle (1.6pt);
\draw [fill=black] (5.249940109966908,2.255825093757596) circle (1.6pt);
\end{scriptsize}
\node at (8.7,3.66) {$\to$};

    \end{scope}
    \begin{scope}[xshift = .8\textwidth, scale = .85]
\draw [line width=.8pt] (7.5,3.1339745962155607)-- (8.,4.);
\draw [line width=.8pt] (8.,4.)-- (7.,4.);
\draw [line width=.8pt] (7.,4.)-- (7.5,3.1339745962155607);
\draw [line width=.8pt] (7.5,3.1339745962155607)-- (6.5,3.133974596215561);
\draw [line width=.8pt] (6.5,3.133974596215561)-- (7.,4.);
\draw [line width=.8pt] (7.,4.)-- (6.,4.);
\draw [line width=.8pt] (6.,4.)-- (6.5,3.133974596215561);
\draw [line width=.8pt] (6.,4.)-- (5.5,3.133974596215561);
\draw [line width=.8pt] (5.5,3.133974596215561)-- (6.5,3.133974596215561);
\draw [line width=.8pt] (6.5,3.133974596215561)-- (6.,2.2679491924311224);
\draw [line width=.8pt] (6.,2.2679491924311224)-- (5.5,3.133974596215561);
\begin{scriptsize}
\draw [fill=black] (8.,4.) circle (1.6pt);
\draw [fill=black] (7.,4.) circle (1.6pt);
\draw [fill=black] (7.5,3.1339745962155607) circle (1.6pt);
\draw [fill=black] (6.5,3.133974596215561) circle (1.6pt);
\draw[color=black] (6.565892549850077,2.9389718782765817) node {$j$};
\draw [fill=black] (6.,4.) circle (1.6pt);
\draw[color=black] (5.855297516616688,4.257011052822387) node {$i$};
\draw [fill=black] (6.5,3.133974596215561) circle (1.6pt);
\draw [fill=black] (5.5,3.133974596215561) circle (1.6pt);
\draw [fill=black] (6.,2.2679491924311224) circle (1.6pt);
\end{scriptsize}
\end{scope}
    \end{tikzpicture}
    \caption{2-switch}
    \label{fig:2_switch}
\end{figure}
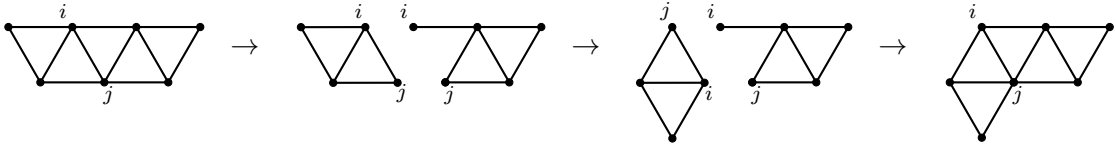

\begin{theorem}\label{thm:2sep2switch}
Let $G$ be a graph with a 2-separator $\{i,j\}$ and $G'$ the graph obtained by performing a 2-switch on $i$ and $j$.  Then $T(G)= T(G')$.  Moreover, if $u$ and $v$ are both in $G_1$ or both in $G_2$, then
\begin{align*}
\F_G(u,v) &= \F_{G'}(u,v) \quad \text{and}\\
r_G(u,v) &= r_{G'}(u,v).
\end{align*}
\end{theorem}
\begin{proof}
The first two equalities follow from Theorems \ref{2sep-tree} and \ref{2sep-forest-same-side} respectively, since they depend only on the smaller graphs, not how they are joined together.  The last follows from the first two and Theorem \ref{thm:2forests/trees}. 
\end{proof}

\begin{example}
After performing a 2-switch in the graph on the left in Figure \ref{fig:2_switch_rd}, we obtain the graph on the right.  Each edge label denotes the number of 2-forests separating the vertices incident to that edge.  Note that these are identical in the subgraph $H$ induced on vertices $\{3,4,5,6,7\}$ and (after a relabel) in the subgraph $K$ induced by $\{1,2,3,4\}$ as guaranteed by Theorem \ref{thm:2sep2switch}.  The number of separating spanning 2-forests is also the same for non-adjacent vertices in $H$ and $K$ (though we have not displayed them).
\end{example}

\begin{figure}[h]
\begin{tikzpicture}
\begin{scope}[scale = 1.7]
\draw [line width=.8pt] (7.5,3.1339745962155607)-- (8.,4.);
\draw [line width=.8pt] (8.,4.)-- (7.,4.);
\draw [line width=.8pt] (7.,4.)-- (7.5,3.1339745962155607);
\draw [line width=.8pt] (7.5,3.1339745962155607)-- (6.5,3.133974596215561);
\draw [line width=.8pt] (6.5,3.133974596215561)-- (7.,4.);
\draw [line width=.8pt] (7.,4.)-- (6.,4.);
\draw [line width=.8pt] (6.,4.)-- (6.5,3.133974596215561);
\draw [line width=.8pt] (6.,4.)-- (5.5,3.133974596215561);
\draw [line width=.8pt] (5.5,3.133974596215561)-- (6.5,3.133974596215561);
\draw [line width=.8pt] (6,4)-- (5,4);
\draw [line width=.8pt] (5,4)-- (5.5,3.133974596215561);
\begin{scriptsize}
\draw [color=black, fill = white] (5,4) circle (3.5pt);
\draw[color=black] (5,4) node {$1$};

\draw [color=black, fill = white]  (5.5,3.133974596215561) circle (3.5pt);
\draw[color=black] (5.5,3.133974596215561) node {$2$};

\draw [color=black, fill = white] (6.5,3.133974596215561) circle (3.5pt);
\draw[color=black] (6.5,3.133974596215561) node {$4$};

\draw [color=black, fill = white] (6.,4.) circle (3.5pt);
\draw[color=black] (6,4) node {$3$};

\draw [color=black, fill = white] (7.,4.) circle (3.5pt);
\draw[color=black] (7,4) node {$5$};

\draw [color=black, fill = white] (7.5,3.1339745962155607) circle (3.5pt);
\draw[color=black] (7.5,3.1339745962155607) node {$6$};

\draw[color=black, fill = white] (8,4) circle (3.5pt);
\draw[color=black] (8,4) node {$7$};

\draw[color=white, fill = white] (7.75,3.535) circle (3.5pt);
\draw[color=black] (7.75,3.565) node {89};

\draw[color=white, fill = white] (7.5,4) circle (3.5pt);
\draw[color=black] (7.5,4) node {89};

\draw[color=white, fill = white] (7.25,3.54) circle (3.5pt);
\draw[color=black] (7.25,3.54) node {68};

\draw[color=white, fill = white] (7,3.13) circle (3.5pt);
\draw[color=black] (7.01,3.13) node {81};

\draw[color=white, fill = white] (6.75,3.54) circle (3.5pt);
\draw[color=black] (6.75,3.54) node {65};

\draw[color=white, fill = white] (6.5,4) circle (3.5pt);
\draw[color=black] (6.5,4) node {80};

\draw[color=white, fill = white] (6.25,3.54) circle (3.5pt);
\draw[color=black] (6.25,3.54) node {65};

\draw[color=white, fill = white] (5.79,3.55) circle (3.5pt);
\draw[color=black] (5.79,3.57) node {68};

\draw[color=white, fill = white] (6,3.13) circle (3.5pt);
\draw[color=black] (6,3.13) node {81};

\draw[color=white, fill = white] (5.5,4) circle (3.5pt);
\draw[color=black] (5.5,4) node {89};

\draw[color=white, fill = white] (5.25,3.54) circle (3.5pt);
\draw[color=black] (5.25,3.54) node {89};
\end{scriptsize}
\node at (8.9,3.5) {\Large $\longrightarrow$};
\node at (6,4.5) {};
\end{scope}

\begin{scope}[xshift = .5\textwidth, scale = 1.7]
\draw [line width=.8pt] (7.5,3.1339745962155607)-- (8.,4.);
\draw [line width=.8pt] (8.,4.)-- (7.,4.);
\draw [line width=.8pt] (7.,4.)-- (7.5,3.1339745962155607);
\draw [line width=.8pt] (7.5,3.1339745962155607)-- (6.5,3.133974596215561);
\draw [line width=.8pt] (6.5,3.133974596215561)-- (7.,4.);
\draw [line width=.8pt] (7.,4.)-- (6.,4.);
\draw [line width=.8pt] (6.,4.)-- (6.5,3.133974596215561);
\draw [line width=.8pt] (6.,4.)-- (5.5,3.133974596215561);
\draw [line width=.8pt] (5.5,3.133974596215561)-- (6.5,3.133974596215561);
\draw [line width=.8pt] (6.5,3.133974596215561)-- (6.,2.2679491924311224);
\draw [line width=.8pt] (6.,2.2679491924311224)-- (5.5,3.133974596215561);
\begin{scriptsize}
\draw [color=black, fill = white] (6.,2.2679491924311224) circle (3.5pt);
\draw[color=black] (6.,2.2679491924311224) node {$1$};

\draw [color=black, fill = white]  (5.5,3.133974596215561) circle (3.5pt);
\draw[color=black] (5.5,3.133974596215561) node {$2$};

\draw [color=black, fill = white] (6.5,3.133974596215561) circle (3.5pt);
\draw[color=black] (6.5,3.133974596215561) node {$4$};

\draw [color=black, fill = white] (6.,4.) circle (3.5pt);
\draw[color=black] (6,4) node {$3$};

\draw [color=black, fill = white] (7.,4.) circle (3.5pt);
\draw[color=black] (7,4) node {$5$};

\draw [color=black, fill = white] (7.5,3.1339745962155607) circle (3.5pt);
\draw[color=black] (7.5,3.1339745962155607) node {$6$};

\draw[color=black, fill = white] (8,4) circle (3.5pt);
\draw[color=black] (8,4) node {$7$};

\draw[color=white, fill = white] (7.75,3.535) circle (3.5pt);
\draw[color=black] (7.75,3.565) node {89};

\draw[color=white, fill = white] (7.5,4) circle (3.5pt);
\draw[color=black] (7.5,4) node {89};

\draw[color=white, fill = white] (7.25,3.54) circle (3.5pt);
\draw[color=black] (7.25,3.54) node {68};

\draw[color=white, fill = white] (7,3.13) circle (3.5pt);
\draw[color=black] (7.01,3.13) node {81};

\draw[color=white, fill = white] (6.75,3.54) circle (3.5pt);
\draw[color=black] (6.75,3.54) node {65};

\draw[color=white, fill = white] (6.5,4) circle (3.5pt);
\draw[color=black] (6.5,4) node {80};

\draw[color=white, fill = white] (6.25,3.54) circle (3.5pt);
\draw[color=black] (6.25,3.54) node {65};

\draw[color=white, fill = white] (5.79,3.55) circle (3.5pt);
\draw[color=black] (5.79,3.57) node {81};

\draw[color=white, fill = white] (6,3.13) circle (3.5pt);
\draw[color=black] (6,3.13) node {68};

\draw[color=white, fill = white] (6.25,2.7) circle (3.5pt);
\draw[color=black] (6.25,2.7) node {89};

\draw[color=white, fill = white] (5.75,2.7) circle (3.5pt);
\draw[color=black] (5.75,2.7) node {89};
\end{scriptsize}

\end{scope}
\end{tikzpicture}
    \caption{Each edge label denotes the number of spanning 2-forests separating the vertices incident to the edge}
    \label{fig:2_switch_rd}

\end{figure}
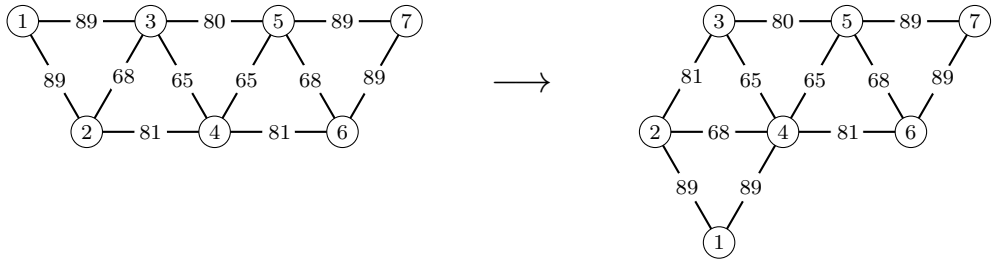

\begin{lemma}\label{obvious}Let $G$ be a graph, and let $x,y,z$ be any vertices of $G$.  Then \[\F_{G}(x,z) = \F_{G}(\{x,y\},z) + \F_{G}(x,\{y,z\}). \]
\end{lemma}
\begin{proof} In any spanning 2-forest that separates $x$ from $z$, the vertex $y$ must be either in the same component as $x$ or in the same component as $z$.
\end{proof}

\begin{theorem}\label{thm:formulation1}
Let $G$ be a graph with a 2-separation, with $i,j$ the two vertices separating the graph, and $G_1, G_2$ the two graphs of the separation.  Let $u\in V(G_1)$ and $v\in V(G_2)$.  Then
\begin{align*}
\F_G(u,v) =& \F_{G_1/ij}(u,ij)T(G_2) +\F_{G_2/ij}(v,ij)T(G_1)\\& + \F_{G_1}(u,i)\F_{G_2}(v,j)+\F_{G_1}(u,j)\F_{G_2}(v,i) \\& - 2\F_{G_1}(u,\{i,j\})\F_{G_2}(v,\{i,j\}).
\end{align*}
\end{theorem}
\begin{proof}
Let $H$ be a spanning 2-forest of $G$ separating $u$ from $v$.  Let $H_1 = H[V(G_1)]$ and $H_2 = H[V(G_2)]$.

\textbf{Case 1:} $H_2$ is a spanning tree of $G_2$.  Then $H_2$ has a single component  containing $v$ and $i$ and $j$, so $H_1$ cannot contain a path connecting $i$ and $j$, since then $H$ would have a cycle.  Thus $H_1$ has three components: one containing $i$, one containing $j$, and one containing $u$.  Now identify vertices $i$ and $j$ in $G_1$, and let $H_1'$ result from identifying $i$ and $j$ in $H_1$.  Then $H_1'$ is a spanning 2-forest of $G_1/ij$ separating $u$ from the vertex $ij$.  By the multiplication principle, the number of possible $H$ in this case is thus $\F_{G_1/ij}(u,ij)T(G_2)$.

\textbf{Case 2:} $H_1$ is a spanning tree of $G_1$.  Note that this case is completely disjoint from Case 1, since if both $H_1$ and $H_2$ were spanning trees of $G_1$ and $G_2$ respectively, then $H$ would not separate $u$ and $v$.  By an argument symmetric to Case 1, the number of spanning 2-forests arising in this case is $\F_{G_2/ij}(v,ij)T(G_1)$.

\textbf{Case 3:} Neither $H_1$ nor $H_2$ is a spanning tree of $G_1$ or $G_2$.  

\textbf{Claim 1:} The number of spanning 2-forests for Case 3 is given by
\begin{align*}
&\F_{G_1}(\{u,i\},j)\F_{G_2}(\{v,j\},i)+ \F_{G_1}(\{u,j\},i)\F_{G_2}(\{v,i\},j)\\&+\F_{G_1}(u,\{i,j\})[\F_{G_2}(\{v,i\},j)+ \F_{G_2}(\{v,j\},i)] \\&+ \F_{G_2}(v,\{i,j\})[\F_{G_1}(\{u,i\},j)+ \F_{G_1}(\{u,j\},i)].\\
\end{align*}
\begin{proof}
Since $H$ is a spanning 2-forest of $G$ separating $u$ from $v$, then we have the following possibilities: $H$ separates $\{u,i\}$ from $\{v,j\}$, $H$ separates $\{u,j\}$ from $\{v,i\}$, $H$ separates $\{u,i,j\}$ from $v$, or $H$ separates $\{v,i,j\}$ from $u$.  Now we consider how $H$ restricts to $G_1$ and $G_2$.  

In the possibility where $\{u,i\}$ and $\{v,j\}$ are separated, $H$ restricted to $G_1$ is a spanning 2-forest of $G_1$ separating $\{u,i\}$ and $j$, while $H$ restricted to $G_2$ is a spanning 2-forest of $G_2$ separating $\{v,j\}$ and $i$.  This yields the first term of the claim. 

The possibility where $\{u,j\}$ and $\{v,i\}$ are separated is symmetric and yields the second term of the claim.

In the remaining possibilities, one of the vertices $u$ or $v$, is separated from the other three.  We consider the possibility where $u$ is separated from $\{i,j,v\}$ and note that the possibility where $v$ is separated from $\{i,j,u\}$ is symmetric to it.

Let $H_2$ be the restriction of $H$ to $G_2$. Since we are in Case 3, $H_2$ is not a spanning tree of $G_2$ and has 2 components.  Let $x$ be a vertex of $H_2$ not in the component of $v$.  The unique path from $x$ to $v$ in $H$ contains an edge of $G_1$.  Since $\{i,j\}$ is a 2-separator, the path from $x$ to $v$ contains both $i$ and $j$.  Thus, either $i$ or $j$ is in a different component than $v$.  So, $H_2$ is either a spanning 2-forest of $G_2$ which separates $v,i$ from $j$, or separates $v,j$ from $i$.  Since there is not a path from $i$ to $j$ in $H_2$, there is a path from $i$ to $j$ in the restriction of $H$ to $G_1$.  So, the restriction of $H$ to $G_1$ is a 2-forest of $G_1$ separating $u$ and $\{i,j\}$. This yields the third term of the claim.

The fourth term corresponds to the symmetric possibility of $v$ being separated from $\{i,j,u\}$ in $H$.
\end{proof}

A substantial application of Theorem \ref{thm:formulation1} is found in Section \ref{sec:applications}.

Applying Lemma \ref{obvious} to the third and fourth terms in the right-hand side of the formula claimed in the theorem in every applicable instance, taking $y$ to be equal to either $i$ or $j$ as appropriate, we find that 
\[\F_{G_1}(u,i)\F_{G_2}(v,j)+\F_{G_1}(u,j)\F_{G_2}(v,i) - 2\F_{G_1}(u,\{i,j\})\F_{G_2}(v,\{i,j\})\] simplifies to the formula of Claim 1.  This completes the proof.
\end{proof}

\begin{lemma}\label{separate2from1}
Let $G$ be any graph, and let $x,y,z$ be any vertices of $G$.  Then
\[
\F_G(\{x,y\},z) = \frac{\F_G(x,z)+\F_G(y,z)-\F_G(x,y)}{2}
\]
\end{lemma}
\begin{proof}
Interchanging the roles of $x,y,z$ in Lemma \ref{obvious}, we get the system
\begin{align*}
\F_{G}(x,z) &= \F_{G}(\{x,y\},z) + \F_{G}(x,\{y,z\})\\
\F_{G}(y,z) &= \F_{G}(\{x,y\},z) + \F_{G}(y,\{x,z\})\\
\F_{G}(x,y) &= \F_{G}(\{x,z\},y) + \F_{G}(x,\{y,z\}).
\end{align*}
Solving this system yields the desired result.
\end{proof}

From Lemma \ref{separate2from1} we obtain an alternative form of the formula in Theorem \ref{thm:formulation1}.  This eliminates counting spanning 2-forests that separate a vertex from a pair of vertices.

\begin{corollary}\label{cor:16}
Let $G$ be a graph with a 2-separation, with $i,j$ the two vertices separating the graph, and $G_1, G_2$ the two graphs of the separation.  Let $u\in V(G_1)$ and $v\in V(G_2)$.  Then
\begin{align*}
\F_G(u,v) =& \F_{G_1/ij}(u,ij)T(G_2) + \F_{G_2/ij}(v,ij)T(G_1)\\& + \frac12\F_{G_1}(u,i)\F_{G_2}(v,j)+\frac12\F_{G_1}(u,j)\F_{G_2}(v,i)\\&-\frac12\F_{G_1}(u,i)\F_{G_2}(v,i)-\frac12\F_{G_1}(u,j)\F_{G_2}(v,j)\\&+\frac12(\F_{G_1}(u,i)+\F_{G_1}(u,j))\F_{G_2}(i,j)\\&+\frac12\F_{G_1}(i,j)(\F_{G_2}(v,i)+\F_{G_2}(v,j))\\&-\frac12\F_{G_1}(i,j)\F_{G_2}(i,j).
\end{align*}
\end{corollary}

\begin{lemma}\label{tree2forest}~\cite[Corollary 4.2]{lovaszrw} Let $i,j$ be vertices of a graph $G$.  Then  $T(G/ij)=F_{G}(i,j)$.
\end{lemma}

The following theorem is found in~\cite[Theorem 4.5]{YangKlein}.  (To our knowledge no combinatorial proof has been given.)

\begin{theorem}\label{identification} Let  $u,v,i,j$ be vertices of a graph $G$.  Then 

$$r_{G/ij}(u,v)=r_G(u,v)-\dfrac{[r_G(u,i)+r_G(v,j)-r_G(u,j)-r_G(v,i)]^2}{4r_G(i,j)}.$$

\end{theorem}

\begin{theorem}\label{thm:sameside} Let $G$ be a graph with a 2-separation, with $i,j$ the two vertices separating the graph, and $G_1, G_2$ the two graphs of the separation.   If $u,v$ are in $G_1$, then 
\[
r_G(u,v)=r_{G_1}(u,v) -\dfrac{[r_{G_1}(u,i)+r_{G_1}(v,j)-r_{G_1}(u,j)-r_{G_1}(v,i)]^2}{4[r_{G_1}(i,j)+r_{G_2}(i,j)]}.
\]
\end{theorem}

\begin{proof}
Using Theorems \ref{thm:2forests/trees}, \ref{2sep-tree}, and \ref{2sep-forest-same-side}, we have

\begin{align*}
r_G(u,v)=&\dfrac{\F_G(u,v)}{T(G)}=\dfrac{\F_{G_1}(u,v)\F_{G_2}(i,j) + \F_{G_1/ij}(u,v)T(G_2)}{T(G_1)\F_{G_2}(i,j) + T(G_2)\F_{G_1}(i,j)}\\  =& \dfrac{\F_{G_1}(u,v)\F_{G_2}(i,j) + \F_{G_1/ij}(u,v)T(G_2)}{T(G_1) T(G_2)[r_{G_2}(i,j)+r_{G_1}(i,j)]}\\ =& \dfrac{r_{G_1}(u,v)\ r_{G_2}(i,j)}{r_{G_1}(i,j)+r_{G_2}(i,j)}+\dfrac{r_{G_1/ij}(u,v)\ T(G_1/ij)}{T(G_1)(r_{G_1}(i,j)+r_{G_2}(i,j))}.
\end{align*}
 
By Lemma \ref{tree2forest}, $T(G_1/ij)=\F_{G_1}(i,j)$.  So we have
\[
r_G(u,v)=\dfrac{r_{G_1}(u,v)\ r_{G_2}(i,j)}{r_{G_1}(i,j)+r_{G_2}(i,j)}+\dfrac{r_{G_1/ij}(u,v)\ r_{G_1}(i,j)}{r_{G_1}(i,j)+r_{G_2}(i,j)}.
\]

Using Theorem \ref{identification}, we replace $r_{G_1/ij}(u,v)$, and arrive at the desired result.
\end{proof}

In order to obtain an analogous formula to Theorem \ref{thm:sameside} in the case that $u$ is in $G_1$ and $v$ is in $G_2$, we divide both sides of the identity in Corollary \ref{cor:16} by $T(G)$ and make use of Theorems \ref{thm:2forests/trees}, \ref{2sep-tree}, and Lemma \ref{tree2forest} to obtain the following.

\begin{theorem}  Let $G$ be a graph with a 2-separation, with $i,j$ the two vertices separating the graph, and $G_1, G_2$ the two graphs of the separation.   If $u$ is in $G_1$ and $v$ is in $G_2$, then 
\[
r_G(u,v)=\dfrac{NUM}{[r_{G_1}(i,j)+r_{G_2}(i,j)]} \\ \text{ where }
\]
\begin{align*}
NUM=& r_{G_1/ij}(u,ij)\ r_{G_1}(i,j) + r_{G_2/ij}(v,ij)\ r_{G_2}(i,j))\\& + \frac12 r_{G_1}(u,i)\ r_{G_2}(v,j)+\frac12 r_{G_1}(u,j)\ r_{G_2}(v,i)\\&-\frac12r_{G_1}(u,i)\ r_{G_2}(v,i)-\frac12r_{G_1}(u,j)\ r_{G_2}(v,j)\\&+\frac12(r_{G_1}(u,i)+r_{G_1}(u,j))\ r_{G_2}(i,j)\\&+\frac12r_{G_1}(i,j)(r_{G_2}(v,i)+r_{G_2}(v,j))\\&-\frac12r_{G_1}(i,j)\ r_{G_2}(i,j).
\end{align*}
\end{theorem}
\section{Applications to 2-connected graphs}\label{sec:applications}

\subsection{Resistance distance and spanning 2-forests in the Sierpinski triangle}
We begin by showing inductively that the resistance distance $r_n(a,b)$ between two vertices of degree 2 in the stage $n$ Sierpinski triangle $S_n$ is $\frac{2}{3}\left(\frac{5}{3}\right)^n$. 
The first three stages are shown in Figure \ref{fig:Sn_building}. 

\begin{figure}[h!]
    \centering
\begin{tikzpicture}[line cap=round,line join=round,>=triangle 45,x=.7cm,y=.7cm]
\draw [line width=1.6pt] (2.,3.4641016151377553)-- (0.,0.);
\draw [line width=1.6pt] (0.,0.)-- (4.,0.);
\draw [line width=1.6pt] (4.,0.)-- (2.,3.4641016151377553);
\begin{scriptsize}
\draw [fill=black] (0.,0.) circle (.8pt);
\draw [fill=black] (4.,0.) circle (.8pt);
\draw [fill=black] (2.,3.4641016151377553) circle (.8pt);
\end{scriptsize}
\node at (2,-.5) {$S_0$};
\end{tikzpicture}
%
\begin{tikzpicture}[line cap=round,line join=round,>=triangle 45,x=.7cm,y=.7cm]
\draw [line width=1.6pt] (2.,3.4641016151377553)-- (0.,0.);
\draw [line width=1.6pt] (0.,0.)-- (4.,0.);
\draw [line width=1.6pt] (4.,0.)-- (2.,3.4641016151377553);
\draw [line width=1.6pt] (1.,1.7320508075688776)-- (3.,1.7320508075688776);
\draw [line width=1.6pt] (3.,1.7320508075688776)-- (2.,0.);
\draw [line width=1.6pt] (2.,0.)-- (1.,1.7320508075688776);
\draw [fill=black] (0.,0.) circle (.8pt);
\draw [fill=black] (4.,0.) circle (.8pt);
\draw [fill=black] (2.,3.4641016151377553) circle (.8pt);
\draw [fill=black] (1.,1.7320508075688776) circle (.8pt);
\draw [fill=black] (3.,1.7320508075688776) circle (.8pt);
\draw [fill=black] (2.,0.) circle (.8pt);
\node at (2,-.5) {$S_1$};
\end{tikzpicture}
%
\begin{tikzpicture}[line cap=round,line join=round,>=triangle 45,x=.7cm,y=.7cm]
\draw [line width=1.6pt] (2.,3.4641016151377553)-- (0.,0.);
\draw [line width=1.6pt] (0.,0.)-- (4.,0.);
\draw [line width=1.6pt] (4.,0.)-- (2.,3.4641016151377553);
\draw [line width=1.6pt] (1.,1.7320508075688776)-- (3.,1.7320508075688776);
\draw [line width=1.6pt] (3.,1.7320508075688776)-- (2.,0.);
\draw [line width=1.6pt] (2.,0.)-- (1.,1.7320508075688776);
\draw [line width=1.6pt] (1.5,2.5980762113533165)-- (2.5,2.5980762113533165);
\draw [line width=1.6pt] (2.5,2.5980762113533165)-- (2.,1.7320508075688776);
\draw [line width=1.6pt] (2.,1.7320508075688776)-- (1.5,2.5980762113533165);
\draw [line width=1.6pt] (0.5,0.8660254037844388)-- (1.5,0.8660254037844388);
\draw [line width=1.6pt] (1.5,0.8660254037844388)-- (1.,0.);
\draw [line width=1.6pt] (1.,0.)-- (0.5,0.8660254037844388);
\draw [line width=1.6pt] (2.5,0.8660254037844388)-- (3.5,0.8660254037844388);
\draw [line width=1.6pt] (3.5,0.8660254037844388)-- (3.,0.);
\draw [line width=1.6pt] (3.,0.)-- (2.5,0.8660254037844388);
\begin{scriptsize}
\draw [fill=black] (0.,0.) circle (.8pt);
\draw [fill=black] (4.,0.) circle (.8pt);
\draw [fill=black] (2.,3.4641016151377553) circle (.8pt);
\draw [fill=black] (1.,1.7320508075688776) circle (.8pt);
\draw [fill=black] (3.,1.7320508075688776) circle (.8pt);
\draw [fill=black] (2.,0.) circle (.8pt);
\draw [fill=black] (0.5,0.8660254037844388) circle (.8pt);
\draw [fill=black] (1.5,2.5980762113533165) circle (.8pt);
\draw [fill=black] (2.5,2.5980762113533165) circle (.8pt);
\draw [fill=black] (3.5,0.8660254037844388) circle (.8pt);
\draw [fill=black] (3.,0.) circle (.8pt);
\draw [fill=black] (1.,0.) circle (.8pt);
\draw [fill=black] (1.5,0.8660254037844388) circle (.8pt);
\draw [fill=black] (2.,1.7320508075688776) circle (.8pt);
\draw [fill=black] (2.5,0.8660254037844388) circle (.8pt);
\end{scriptsize}
\node at (2,-.5) {$S_2$};
\end{tikzpicture}
    \caption{The first three stages of the Sierpinski triangle, $S_0$, $S_1$, and $S_2$}
    \label{fig:Sn_building}
\end{figure}
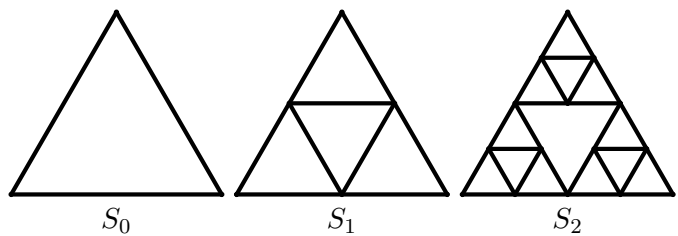

\begin{figure}[h!]
    \begin{center}
\begin{tikzpicture}[line cap=round,line join=round,>=triangle 45,x=.7cm,y=.7cm]
\draw [line width=1.6pt] (2.,3.4641016151377553)-- (0.,0.);
\draw [line width=1.6pt] (0.,0.)-- (4.,0.);
\draw [line width=1.6pt] (4.,0.)-- (2.,3.4641016151377553);
\draw [line width=1.6pt] (1.,1.7320508075688776)-- (3.,1.7320508075688776);
\draw [line width=1.6pt] (3.,1.7320508075688776)-- (2.,0.);
\draw [line width=1.6pt] (2.,0.)-- (1.,1.7320508075688776);
\draw [line width=1.6pt] (1.5,2.5980762113533165)-- (2.5,2.5980762113533165);
\draw [line width=1.6pt] (2.5,2.5980762113533165)-- (2.,1.7320508075688776);
\draw [line width=1.6pt] (2.,1.7320508075688776)-- (1.5,2.5980762113533165);
\draw [line width=1.6pt] (0.5,0.8660254037844388)-- (1.5,0.8660254037844388);
\draw [line width=1.6pt] (1.5,0.8660254037844388)-- (1.,0.);
\draw [line width=1.6pt] (1.,0.)-- (0.5,0.8660254037844388);
\draw [line width=1.6pt] (2.5,0.8660254037844388)-- (3.5,0.8660254037844388);
\draw [line width=1.6pt] (3.5,0.8660254037844388)-- (3.,0.);
\draw [line width=1.6pt] (3.,0.)-- (2.5,0.8660254037844388);
\begin{scriptsize}
\draw [fill=black] (0.,0.) circle (.8pt);
\draw [fill=black] (4.,0.) circle (.8pt);
\draw [fill=black] (2.,3.4641016151377553) circle (.8pt);
\draw [fill=black] (1.,1.7320508075688776) circle (.8pt);
\draw [fill=black] (3.,1.7320508075688776) circle (.8pt);
\draw [fill=black] (2.,0.) circle (.8pt);
\draw [fill=black] (0.5,0.8660254037844388) circle (.8pt);
\draw [fill=black] (1.5,2.5980762113533165) circle (.8pt);
\draw [fill=black] (2.5,2.5980762113533165) circle (.8pt);
\draw [fill=black] (3.5,0.8660254037844388) circle (.8pt);
\draw [fill=black] (3.,0.) circle (.8pt);
\draw [fill=black] (1.,0.) circle (.8pt);
\draw [fill=black] (1.5,0.8660254037844388) circle (.8pt);
\draw [fill=black] (2.,1.7320508075688776) circle (.8pt);
\draw [fill=black] (2.5,0.8660254037844388) circle (.8pt);
\end{scriptsize}
\node at (.5,2.4) {$G_1$};
\begin{small}
\draw[color=black] (-0.32,-0.03) node {$a$};
\draw[color=black] (4.26,-0.07) node {$b$};
\draw[color=black] (1.96,3.79) node {$c$};
\end{small}
\end{tikzpicture}\\[5mm]
\begin{minipage}{6.5cm}
\begin{tikzpicture}[line cap=round,line join=round,>=triangle 45,x=.7cm,y=.7cm]
\draw [line width=1.6pt] (2.,3.4641016151377553)-- (0.,0.);
\draw [line width=1.6pt] (0.,0.)-- (4.,0.);
\draw [line width=1.6pt] (4.,0.)-- (2.,3.4641016151377553);
\draw [line width=1.6pt] (1.,1.7320508075688776)-- (3.,1.7320508075688776);
\draw [line width=1.6pt] (3.,1.7320508075688776)-- (2.,0.);
\draw [line width=1.6pt] (2.,0.)-- (1.,1.7320508075688776);
\draw [line width=1.6pt] (1.5,2.5980762113533165)-- (2.5,2.5980762113533165);
\draw [line width=1.6pt] (2.5,2.5980762113533165)-- (2.,1.7320508075688776);
\draw [line width=1.6pt] (2.,1.7320508075688776)-- (1.5,2.5980762113533165);
\draw [line width=1.6pt] (0.5,0.8660254037844388)-- (1.5,0.8660254037844388);
\draw [line width=1.6pt] (1.5,0.8660254037844388)-- (1.,0.);
\draw [line width=1.6pt] (1.,0.)-- (0.5,0.8660254037844388);
\draw [line width=1.6pt] (2.5,0.8660254037844388)-- (3.5,0.8660254037844388);
\draw [line width=1.6pt] (3.5,0.8660254037844388)-- (3.,0.);
\draw [line width=1.6pt] (3.,0.)-- (2.5,0.8660254037844388);
\begin{scriptsize}
\draw [fill=black] (0.,0.) circle (.8pt);
\draw [fill=black] (4.,0.) circle (.8pt);
\draw [fill=black] (2.,3.4641016151377553) circle (.8pt);
\draw [fill=black] (1.,1.7320508075688776) circle (.8pt);
\draw [fill=black] (3.,1.7320508075688776) circle (.8pt);
\draw [fill=black] (2.,0.) circle (.8pt);
\draw [fill=black] (0.5,0.8660254037844388) circle (.8pt);
\draw [fill=black] (1.5,2.5980762113533165) circle (.8pt);
\draw [fill=black] (2.5,2.5980762113533165) circle (.8pt);
\draw [fill=black] (3.5,0.8660254037844388) circle (.8pt);
\draw [fill=black] (3.,0.) circle (.8pt);
\draw [fill=black] (1.,0.) circle (.8pt);
\draw [fill=black] (1.5,0.8660254037844388) circle (.8pt);
\draw [fill=black] (2.,1.7320508075688776) circle (.8pt);
\draw [fill=black] (2.5,0.8660254037844388) circle (.8pt);
\end{scriptsize}
\node at (.5,2.4) {$G_2$};
\begin{small}
\draw[color=black] (-0.32,-0.03) node {$d$};
\draw[color=black] (3.9,-0.5) node {$e$};
\draw[color=black] (1.96,3.79) node {$a$};
\end{small}

\end{tikzpicture}

\vspace{-3.47cm}

\hspace*{3.15cm}
\begin{tikzpicture}[line cap=round,line join=round,>=triangle 45,x=.7cm,y=.7cm]
\draw [line width=1.6pt] (2.,3.4641016151377553)-- (0.,0.);
\draw [line width=1.6pt] (0.,0.)-- (4.,0.);
\draw [line width=1.6pt] (4.,0.)-- (2.,3.4641016151377553);
\draw [line width=1.6pt] (1.,1.7320508075688776)-- (3.,1.7320508075688776);
\draw [line width=1.6pt] (3.,1.7320508075688776)-- (2.,0.);
\draw [line width=1.6pt] (2.,0.)-- (1.,1.7320508075688776);
\draw [line width=1.6pt] (1.5,2.5980762113533165)-- (2.5,2.5980762113533165);
\draw [line width=1.6pt] (2.5,2.5980762113533165)-- (2.,1.7320508075688776);
\draw [line width=1.6pt] (2.,1.7320508075688776)-- (1.5,2.5980762113533165);
\draw [line width=1.6pt] (0.5,0.8660254037844388)-- (1.5,0.8660254037844388);
\draw [line width=1.6pt] (1.5,0.8660254037844388)-- (1.,0.);
\draw [line width=1.6pt] (1.,0.)-- (0.5,0.8660254037844388);
\draw [line width=1.6pt] (2.5,0.8660254037844388)-- (3.5,0.8660254037844388);
\draw [line width=1.6pt] (3.5,0.8660254037844388)-- (3.,0.);
\draw [line width=1.6pt] (3.,0.)-- (2.5,0.8660254037844388);
\begin{scriptsize}
\draw [fill=black] (0.,0.) circle (.8pt);
\draw [fill=black] (4.,0.) circle (.8pt);
\draw [fill=black] (2.,3.4641016151377553) circle (.8pt);
\draw [fill=black] (1.,1.7320508075688776) circle (.8pt);
\draw [fill=black] (3.,1.7320508075688776) circle (.8pt);
\draw [fill=black] (2.,0.) circle (.8pt);
\draw [fill=black] (0.5,0.8660254037844388) circle (.8pt);
\draw [fill=black] (1.5,2.5980762113533165) circle (.8pt);
\draw [fill=black] (2.5,2.5980762113533165) circle (.8pt);
\draw [fill=black] (3.5,0.8660254037844388) circle (.8pt);
\draw [fill=black] (3.,0.) circle (.8pt);
\draw [fill=black] (1.,0.) circle (.8pt);
\draw [fill=black] (1.5,0.8660254037844388) circle (.8pt);
\draw [fill=black] (2.,1.7320508075688776) circle (.8pt);
\draw [fill=black] (2.5,0.8660254037844388) circle (.8pt);
\end{scriptsize}
\begin{small}
\draw[color=black] (4.26,-0.07) node {$f$};
\draw[color=black] (1.96,3.79) node {$b$};
\end{small}

\end{tikzpicture}
\end{minipage}
\end{center}
    \caption{A 2-separation of $S_n$}
    \label{fig:Sn_sep}
\end{figure}
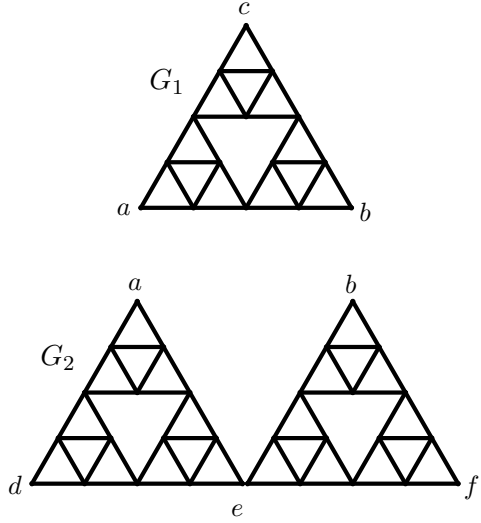

It is straightforward to verify that $r_0(a,b) = 2/3$. 
Now, we assume that $r_n(a,b) = x$, and use the separation shown in Figure \ref{fig:Sn_sep} together with Theorem \ref{thm:sameside} to show that 
\[\begin{array}{c}
\displaystyle{
r_{n+1}(d,f) = r_{G_2}(d,f) - \frac{(r_{G_2}(d,a)+r_{G_2}(f,b) - r_{G_2}(d,b) - r_{G_2}(f,a))^2}{4(r_{G_1}(a,b) + r_{G_2}(a,b))}
}\\[5mm]
\displaystyle{=2x - \frac{(x + x - 2x - 2x)^2}{4(x + 2x)} }\\[5mm] 
\displaystyle{=2x - \frac{4x^2}{4(3x)} = 2x - \frac{1}{3}x = \frac{5}{3}x
}.
\end{array}
\]
Thus, $r_n(a,b) = \frac{2}{3}\left(\frac{5}{3}\right)^n$, as desired. 

This result is particularly surprising, as it was shown in \cite{Sierp} that the number of spanning trees in $S_n$ is 
\begin{equation}\label{eq:STinSn}
2^{\frac{1}{2}(3^n-1)}3^{\frac{1}{4}(3^{n+1}+2n+1)}5^{\frac{1}{4}(3^n-2n-1)},
\end{equation}
which would seem to suggest that formulae for the resistance distances in $S_n$ would be at least as complicated. In fact, Equation \ref{eq:STinSn} together with Theorem \ref{thm:2forests/trees} give the number of spanning 2-forests which separate two degree-2 vertices in $S_n$ as
\[
\mathcal{F}_{S_n}(a,b) = 2^{\frac{1}{2}(3^n+1)}3^{\frac{1}{4}(3^{n+1}-2n-3)}5^{\frac{1}{4}(3^n+2n-1)} .
\]

We also point out that $r_n(a,b) \to \infty$ as $n \to \infty$.

\subsection{Resistance distance in a bent linear 2-tree}

Earlier work by the authors considered resistance distance in a straight linear 2-tree with $n$ vertices and obtained the following result.
\begin{theorem}\label{thm:sl2t}
Let $H_n$ be the straight linear 2-tree on $n$ vertices labeled as in Figure~\ref{fig:2tree}, and let $m = n-2$ be the number of triangles in $H_n$.  Then for any two vertices $j$ and $j+k$ of $H_n$, 
{\begin{equation}
\begin{aligned}  r_{H_n}(j,j+k) &= \frac{F_{m+1}^2+F_k^2F_{m-2j-k+3}^2}{F_{2m+2}} \label{eq:rdsl2t} \\&+\frac{\frac{F_{m+1}}{5}\left[F_{m-k}(kL_k-F_k)+F_{m-k+1}\left((k-5)F_{k+1}+(2k+2)F_{k}\right)\right]}{F_{2m+2}}   \end{aligned}
 \end{equation}}
\noindent where $F_p$ is the $p$th Fibonacci number and $L_k$ is the $k$th Lucas number. 
\end{theorem}

Modifying this formula by replacing $j$ by $u$, $j+k$ by $v$, $m$ by $n-2$, then multiplying by $F_{2n-2}$ and applying 
Theorem~\ref{thm:2forests/trees}  we have

\begin{equation}\label{eq:rdsl2tuv}\begin{aligned}\F_{H_n}(u,v)&= F_{n-1}^2+F_{v-u}^2F_{n-u-v+1}^2\\ 
	 &\quad + \frac{F_{n-1}}5 [F_{n+u-v-2}((v-u)L_{v-u}-F_{v-u}) \\
	 &\quad + F_{n+u-v-1}((v-u-5)F_{v-u+1}+2(v-u+1)F_{v-u}]. \end{aligned}\end{equation}    

In ~\cite{bef} we also gave an alternative formula for $r_{H_n}(j,j+k)$.  Modifying it in the same way yields

\begin{equation}
\F_{H_n}(u,v)=\sum_{i=1}^{v-u} (F_i F_{i+2u-2}-F_{i-1} F_{i+2u-3})F_{2n-2i-2u+1}. \label{eq:resdiststraightsum}
\end{equation}
Although it is not a closed form expression, is is nevertheless useful as we shall see.

In this work we consider a modification of the straight linear 2-tree which we term the bent linear 2-tree whose definition is below.

\begin{definition}[bent linear 2-tree]\label{def:lin2treestb}
We define the graph $G_n$ with $V(G_n)= V(H_n)$ and $E(G_n) = (E(H_n) \cup \{k,k+3\})\setminus(\{k+1,k+3\})$ to be a bent linear 2-tree with bend at vertex $k$.
 See Figure~\ref{fig:bent}.
\end{definition}
In essence a bent linear 2 tree differs from a straight linear 2 tree in that vertex $k$ has degree 5, vertex $k+1$ has degree $3$ and all other vertices have degrees as before.

\begin{figure}\label{fig:bent2tree}
\begin{center}
\begin{tikzpicture}[line cap=round,line join=round,>=triangle 45,x=1.0cm,y=1.0cm,scale = 1.2]
\draw [line width=.8pt] (-5.464101615137757,2.)-- (-4.598076211353318,1.5);
\draw [line width=.8pt] (-4.598076211353318,1.5)-- (-4.598076211353318,2.5);
\draw [line width=.8pt] (-4.598076211353318,2.5)-- (-3.732050807568879,2.);
\draw [line width=.8pt] (-3.732050807568879,2.)-- (-3.7320508075688785,3.);
\draw [line width=.8pt,dotted] (-3.7320508075688785,3.)-- (-2.8660254037844393,2.5);
\draw [line width=.8pt] (-2.8660254037844393,2.5)-- (-2.866025403784439,3.5);
\draw [line width=.8pt] (-2.866025403784439,3.5)-- (-2.,3.);
\draw [line width=.8pt] (-2.,3.)-- (-2.,4.);
\draw [line width=.8pt] (-2.,4.)-- (-1.1339745962155612,3.5);
\draw [line width=.8pt] (-1.1339745962155612,3.5)-- (-2.,3.);
\draw [line width=.8pt] (-2.,3.)-- (-1.1339745962155616,2.5);
\draw [line width=.8pt] (-1.1339745962155616,2.5)-- (-1.1339745962155612,3.5);
\draw [line width=.8pt] (-1.1339745962155612,3.5)-- (-0.2679491924311225,3.);
\draw [line width=.8pt] (-0.2679491924311225,3.)-- (-1.1339745962155616,2.5);
\draw [line width=.8pt,dotted] (-1.1339745962155616,2.5)-- (-0.26794919243112303,2.);
\draw [line width=.8pt,dotted] (-0.26794919243112303,2.)-- (-0.2679491924311225,3.);
\draw [line width=.8pt,dotted] (-0.2679491924311225,3.)-- (0.5980762113533165,2.5);
\draw [line width=.8pt] (0.5980762113533165,2.5)-- (-0.26794919243112303,2.);
\draw [line width=.8pt] (-0.26794919243112303,2.)-- (0.5980762113533152,1.5);
\draw [line width=.8pt] (0.5980762113533152,1.5)-- (0.5980762113533165,2.5);
\draw [line width=.8pt] (0.5980762113533165,2.5)-- (1.464101615137755,2.);
\draw [line width=.8pt] (1.464101615137755,2.)-- (0.5980762113533152,1.5);
\draw [line width=.8pt] (-2.,4.)-- (-2.866025403784439,3.5);
\draw [line width=.8pt,dotted] (-2.866025403784439,3.5)-- (-3.7320508075688785,3.);
\draw [line width=.8pt] (-3.7320508075688785,3.)-- (-4.598076211353318,2.5);
\draw [line width=.8pt] (-4.598076211353318,2.5)-- (-5.464101615137757,2.);
\draw [line width=.8pt] (-4.598076211353318,1.5)-- (-3.732050807568879,2.);
\draw [line width=.8pt,dotted] (-3.732050807568879,2.)-- (-2.8660254037844393,2.5);
\draw [line width=.8pt] (-2.8660254037844393,2.5)-- (-2.,3.);
\begin{scriptsize}
\draw [fill=black] (-2.,4.) circle (1.5pt);
\draw[color=black] (-2.06648828953202,4.259331085745072) node {$k+1$};
\draw [fill=black] (-1.1339745962155612,3.5) circle (1.5pt);
\draw[color=black] (-0.9407359844212153,3.7428094398707032) node {$k+2$};
\draw [fill=black] (-2.,3.) circle (1.5pt);
\draw[color=black] (-2.0267558552339913,2.6989231016718021) node {$k$};
\draw [fill=black] (-2.866025403784439,3.5) circle (1.5pt);
\draw[color=black] (-3.0597991469827295,3.689832860806665) node {$k-1$};
\draw [fill=black] (-2.8660254037844393,2.5) circle (1.5pt);
\draw[color=black] (-2.702207238300474,2.2919066737377372) node {$k-2$};
\draw [fill=black] (-1.1339745962155616,2.5) circle (1.5pt);
\draw[color=black] (-1.3248161826354898,2.2919066737377372) node {$k+3$};
\draw [fill=black] (-2.,3.) circle (1.5pt);
\draw [fill=black] (-0.2679491924311225,3.) circle (1.5pt);
\draw[color=black] (-0.14608729846064736,3.213043649230324) node {$k+4$};
\draw [fill=black] (-3.7320508075688785,3.) circle (1.5pt);
\draw[color=black] (-3.9339127015393545,3.2395319387623434) node {$5$};
\draw [fill=black] (-3.732050807568879,2.) circle (1.5pt);
\draw[color=black] (-3.6028090823891175,1.8488967383313488) node {$4$};
\draw [fill=black] (-0.26794919243112303,2.) circle (1.5pt);
\draw[color=black] (-0.4374584833128556,1.7488967383313488) node {$n-3$};
\draw [fill=black] (-4.598076211353318,2.5) circle (1.5pt);
\draw[color=black] (-4.78153796656396,2.7494985824199927) node {$3$};
\draw [fill=black] (-4.598076211353318,1.5) circle (1.5pt);
\draw[color=black] (-4.463678492179733,1.345619237222989) node {$2$};
\draw [fill=black] (-5.464101615137757,2.) circle (1.5pt);
\draw[color=black] (-5.655651521120585,2.1270237784175476) node {$1$};
\draw [fill=black] (0.5980762113533165,2.5) circle (1.5pt);
\draw[color=black] (0.7280262560959774,2.709766148121964) node {$n-2$};
\draw [fill=black] (0.5980762113533152,1.5) circle (1.5pt);
\draw[color=black] (0.4234109264777597,1.2721075267550078) node {$n-1$};
\draw [fill=black] (1.464101615137755,2.) circle (1.5pt);
\draw[color=black] (1.628628100184621,2.0475589098214906) node {$n$};
\end{scriptsize}
\end{tikzpicture}
\end{center}
\caption{A linear 2-tree with $n$ vertices and single bend at vertex $k$.}
\label{fig:bent}
\end{figure}
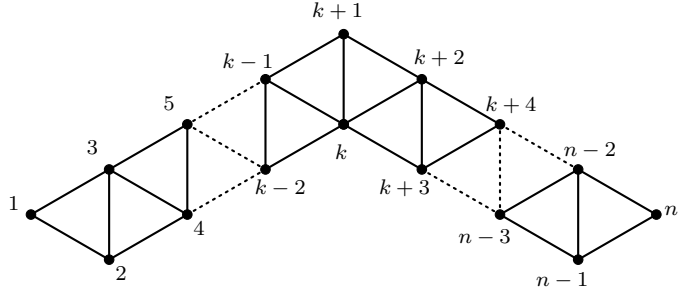

When $u$ and $v$ are on the same side of the bent linear 2-tree it is easy to determine $\F_{G_n}(u,v)$.  Applying Theorem~\ref{thm:2sep2switch} we obtain:

\begin{theorem} 
Let $G_n$ be the bent linear 2-tree on $n$ vertices labeled as in Figure \ref{fig:bent} with a bend located at vertex $k$ and let $H_n$ be the straight linear 2-tree on $n$ vertices labeled as in Figure~\ref{fig:2tree}.  Then for any two vertices $u$ and $v$ of $G_n$, where $u < v \leq  k+1$ or $k+1<u< v \leq n$

\[ \F_{G_n}(u,v)=\F_{H_n}(u,v)  {\ \rm and\ } r_{G_n}(u,v)=r_{H_n}(u,v), \]
where $F_{H_n}(u,v)$ is given by Equation (\ref{eq:rdsl2tuv}).             
							    
\end{theorem}

\noindent For all other $u$ and $v$ in $V(G_n)$ the number of spanning 2-trees and resistance distances are determined as follows.

\begin{theorem}\label{thm:bigone} Let $G_n$ be a bent linear 2-tree on $n$ vertices with a single bend at vertex $k$.  Then for any two vertices $u$ and $v$ with $u \le k+1$ and $v>k+1$, 
\footnotesize
\[
\F_{G_n}(u,v) = \F_{H_n}(u,v) -  [F_{k-2}F_{k+1}+2(-1)^{k-u} F_{u-1}^2] \, [F_{n-k-2}F_{n-k+1} + 2 (-1)^{v-k-1}F_{n-v}^2],
\]
\normalsize
while the resistance distance between $u$ and $v$ is given by
\footnotesize 
\[ r_{G_n} (u,v)= r_{H_n}(u,v) -  [F_{k-2}F_{k+1}+2(-1)^{k-u} F_{u-1}^2] \, [F_{n-k-2}F_{n-k+1} + 2 (-1)^{v-k-1}F_{n-v}^2]/F_{2n-2}. \]
 \normalsize
\end{theorem}
 
\begin{proof}
Let $H$ be the straight linear 2-tree with $n$ vertices and denote $G_n$ by $G$.  We note that both $G$ and $H$ are graphs with a 2-separation $\{k,k+1\}$.  Moreover $G$ is obtained by performing a 2-switch on the vertices $k$ and $k+1$ in $H$.  

From Theorem~\ref{thm:formulation1} we observe that
\begin{align*}
\F_G(u,v) =& \F_{G_1/ij}(u,ij)T(G_2) +\F_{G_2/ij}(v,ij)T(G_1)\\& + \F_{G_1}(u,i)\F_{G_2}(v,j)+\F_{G_1}(u,j)\F_{G_2}(v,i) \\& - 2\F_{G_1}(u,\{i,j\})\F_{G_2}(v,\{i,j\}),
\end{align*}
where $i=k$, $j=k+1$ and $G_1$ and $G_2$ are the two graphs of the 2-separation of $G$ as shown in Figure~\ref{fig:B}.  
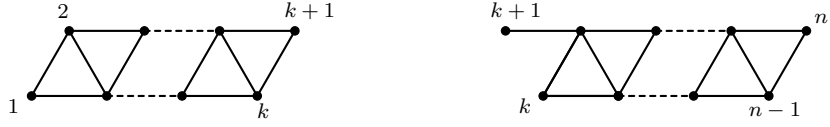
\begin{figure}[!ht]
\begin{center}

\begin{tikzpicture}[line cap=round,line join=round,>=triangle 45,x=1.0cm,y=1.0cm]
\begin{scope}

\draw [line width=.8pt] (-2.,3.)-- (-1.5,3.8660254037844393);
\draw [line width=.8pt] (-1.5,3.8660254037844393)-- (-1.,3.);
\draw [line width=.8pt] (-1.,3.)-- (-2.,3.);
\draw [line width=.8pt] (-1.5,3.8660254037844393)-- (-0.5,3.866025403784439);
\draw [line width=.8pt] (-1.,3.)-- (-0.5,3.866025403784439);
\draw [line width=.8pt,dash pattern=on 2pt off 2pt] (-1.,3.)-- (0.,3.);
\draw [line width=.8pt,dash pattern=on 2pt off 2pt] (-0.5,3.866025403784439)-- (0.5,3.8660254037844375);
\draw [line width=.8pt] (0.5,3.8660254037844375)-- (0.,3.);
\draw [line width=.8pt] (0.,3.)-- (1.,3.);
\draw [line width=.8pt] (1.,3.)-- (0.5,3.8660254037844375);
\draw [line width=.8pt] (0.5,3.8660254037844375)-- (1.5,3.8660254037844357);
\draw [line width=.8pt] (1.5,3.8660254037844357)-- (1.,3.);
\begin{scriptsize}
\draw [fill=black] (-2.,3.) circle (1.6pt);
\draw[color=black] (-2.24,2.87) node {$1$};
\draw [fill=black] (-1.,3.) circle (1.6pt);
\draw [fill=black] (-1.5,3.8660254037844393) circle (1.6pt);
\draw[color=black] (-1.58,4.13) node {$2$};
\draw [fill=black] (-0.5,3.866025403784439) circle (1.6pt);
\draw [fill=black] (0.,3.) circle (1.6pt);
\draw [fill=black] (0.5,3.8660254037844375) circle (1.6pt);
\draw [fill=black] (1.,3.) circle (1.6pt);
\draw[color=black] (1.08,2.81) node {$k$};
\draw [fill=black] (1.5,3.8660254037844357) circle (1.6pt);
\draw[color=black] (1.7,4.13) node {$k+1$};
\end{scriptsize}
\end{scope}
\begin{scope}[xshift = .45\textwidth]
\draw [line width=.8pt] (-2.,3.)-- (-1.5,3.8660254037844393);
\draw [line width=.8pt] (-1.5,3.8660254037844393)-- (-1.,3.);
\draw [line width=.8pt] (-1.,3.)-- (-2.,3.);
\draw [line width=.8pt] (-1.5,3.8660254037844393)-- (-0.5,3.866025403784439);
\draw [line width=.8pt] (-1.,3.)-- (-0.5,3.866025403784439);
\draw [line width=.8pt,dash pattern=on 2pt off 2pt] (-1.,3.)-- (0.,3.);
\draw [line width=.8pt,dash pattern=on 2pt off 2pt] (-0.5,3.866025403784439)-- (0.5,3.8660254037844375);
\draw [line width=.8pt] (0.5,3.8660254037844375)-- (0.,3.);
\draw [line width=.8pt] (0.,3.)-- (1.,3.);
\draw [line width=.8pt] (1.,3.)-- (0.5,3.8660254037844375);
\draw [line width=.8pt] (0.5,3.8660254037844375)-- (1.5,3.8660254037844357);
\draw [line width=.8pt] (1.5,3.8660254037844357)-- (1.,3.);
\draw [line width=.8pt] (-2.5,3.8660254037844397)-- (-1.5,3.8660254037844393);
\draw [line width=.8pt] (-1.5,3.8660254037844393)-- (-2.,3.);
\draw [line width=.8pt] (-2.,3.)-- (-1.,3.);
\begin{scriptsize}
\draw [fill=black] (-2.,3.) circle (1.6pt);
\draw[color=black] (-2.24,2.87) node {$k$};
\draw [fill=black] (-1.,3.) circle (1.6pt);
\draw [fill=black] (-1.5,3.8660254037844393) circle (1.6pt);
\draw [fill=black] (-0.5,3.866025403784439) circle (1.6pt);
\draw [fill=black] (0.,3.) circle (1.6pt);
\draw [fill=black] (0.5,3.8660254037844375) circle (1.6pt);
\draw [fill=black] (1.,3.) circle (1.6pt);
\draw[color=black] (1.08,2.81) node {$n-1$};
\draw [fill=black] (1.5,3.8660254037844357) circle (1.6pt);
\draw[color=black] (1.7,4.03) node {$n$};
\draw [fill=black] (-1.,3.) circle (1.6pt);
\draw [fill=black] (-2.5,3.8660254037844397) circle (1.6pt);
\draw[color=black] (-2.36,4.13) node {$k+1$};
\end{scriptsize}
\end{scope}
\end{tikzpicture}
\end{center}
\caption{The two graphs of the 2-separation of $G$.  Here $G_1$ is on the left and $G_2$ is on the right.}\label{fig:B}
\end{figure}

Similarly we have 
\begin{align*}
\F_H(u,v) =& \F_{H_1/ij}(u,ij)T(H_2) +\F_{H_2/ij}(v,ij)T(H_1)\\& + \F_{H_1}(u,i)\F_{H_2}(v,j)+\F_{H_1}(u,j)\F_{H_2}(v,i) \\& - 2\F_{H_1}(u,\{i,j\})\F_{H_2}(v,\{i,j\}),
\end{align*}
where $i=k$, $j=k+1$ and $H_1$ and $H_2$ are the two graphs of the 2-separation of $H$ as shown in Figure~\ref{fig:BH}.  
\begin{figure}[h!]
\begin{center}

\begin{tikzpicture}[line cap=round,line join=round,>=triangle 45,x=1.0cm,y=1.0cm]
\begin{scope}

\draw [line width=.8pt] (-2.,3.)-- (-1.5,3.8660254037844393);
\draw [line width=.8pt] (-1.5,3.8660254037844393)-- (-1.,3.);
\draw [line width=.8pt] (-1.,3.)-- (-2.,3.);
\draw [line width=.8pt] (-1.5,3.8660254037844393)-- (-0.5,3.866025403784439);
\draw [line width=.8pt] (-1.,3.)-- (-0.5,3.866025403784439);
\draw [line width=.8pt,dash pattern=on 2pt off 2pt] (-1.,3.)-- (0.,3.);
\draw [line width=.8pt,dash pattern=on 2pt off 2pt] (-0.5,3.866025403784439)-- (0.5,3.8660254037844375);
\draw [line width=.8pt] (0.5,3.8660254037844375)-- (0.,3.);
\draw [line width=.8pt] (0.,3.)-- (1.,3.);
\draw [line width=.8pt] (1.,3.)-- (0.5,3.8660254037844375);
\draw [line width=.8pt] (0.5,3.8660254037844375)-- (1.5,3.8660254037844357);
\draw [line width=.8pt] (1.5,3.8660254037844357)-- (1.,3.);
\begin{scriptsize}
\draw [fill=black] (-2.,3.) circle (1.6pt);
\draw[color=black] (-2.24,2.87) node {$1$};
\draw [fill=black] (-1.,3.) circle (1.6pt);
\draw [fill=black] (-1.5,3.8660254037844393) circle (1.6pt);
\draw[color=black] (-1.58,4.13) node {$2$};
\draw [fill=black] (-0.5,3.866025403784439) circle (1.6pt);
\draw [fill=black] (0.,3.) circle (1.6pt);
\draw [fill=black] (0.5,3.8660254037844375) circle (1.6pt);
\draw [fill=black] (1.,3.) circle (1.6pt);
\draw[color=black] (1.08,2.81) node {$k$};
\draw [fill=black] (1.5,3.8660254037844357) circle (1.6pt);
\draw[color=black] (1.7,4.13) node {$k+1$};
\end{scriptsize}
\end{scope}
\begin{scope}[xshift = .45\textwidth]
\draw [line width=.8pt] (-2.,3.)-- (-1.5,3.8660254037844393);
\draw [line width=.8pt] (-1.5,3.8660254037844393)-- (-1.,3.);
\draw [line width=.8pt] (-1.,3.)-- (-2.,3.);
\draw [line width=.8pt] (-1.5,3.8660254037844393)-- (-0.5,3.866025403784439);
\draw [line width=.8pt] (-1.,3.)-- (-0.5,3.866025403784439);
\draw [line width=.8pt,dash pattern=on 2pt off 2pt] (-1.,3.)-- (0.,3.);
\draw [line width=.8pt,dash pattern=on 2pt off 2pt] (-0.5,3.866025403784439)-- (0.5,3.8660254037844375);
\draw [line width=.8pt] (0.5,3.8660254037844375)-- (0.,3.);
\draw [line width=.8pt] (0.,3.)-- (1.,3.);
\draw [line width=.8pt] (1.,3.)-- (0.5,3.8660254037844375);
\draw [line width=.8pt] (0.5,3.8660254037844375)-- (1.5,3.8660254037844357);
\draw [line width=.8pt] (1.5,3.8660254037844357)-- (1.,3.);
\draw [line width=.8pt] (-2.5,3.8660254037844397)-- (-1.5,3.8660254037844393);
\draw [line width=.8pt] (-1.5,3.8660254037844393)-- (-2.,3.);
\draw [line width=.8pt] (-2.,3.)-- (-1.,3.);
\begin{scriptsize}
\draw [fill=black] (-2.,3.) circle (1.6pt);
\draw[color=black] (-2.34,2.77) node {$k+1$};
\draw [fill=black] (-1.,3.) circle (1.6pt);
\draw [fill=black] (-1.5,3.8660254037844393) circle (1.6pt);
\draw [fill=black] (-0.5,3.866025403784439) circle (1.6pt);
\draw [fill=black] (0.,3.) circle (1.6pt);
\draw [fill=black] (0.5,3.8660254037844375) circle (1.6pt);
\draw [fill=black] (1.,3.) circle (1.6pt);
\draw[color=black] (1.08,2.81) node {$n-1$};
\draw [fill=black] (1.5,3.8660254037844357) circle (1.6pt);
\draw[color=black] (1.7,4.03) node {$n$};
\draw [fill=black] (-1.,3.) circle (1.6pt);
\draw [fill=black] (-2.5,3.8660254037844397) circle (1.6pt);
\draw[color=black] (-2.36,4.13) node {$k$};
\end{scriptsize}
\end{scope}
\end{tikzpicture}
\end{center}
\caption{The two graphs of the 2-separation of $H$.  Here $H_1$ is on the left and $H_2$ is on the right.  Observe the difference between the location of nodes $k$ and $k+1$ in $H_2$ compared to $G_2$.}\label{fig:BH}
\end{figure}
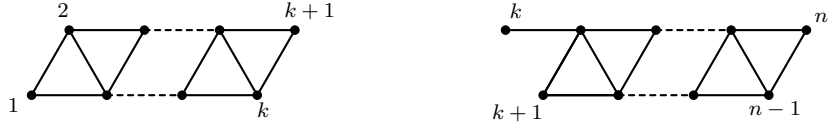

We see by definition that 
\[2\F_{G_1}(u,\{i,j\})\F_{G_2}(v,\{i,j\}) = 2\F_{H_1}(u,\{i,j\})\F_{H_2}(v,\{i,j\}).\] 
It is also clear that
\small 
\[ \F_{G_1/ij}(u,ij)T(G_2) +\F_{G_2/ij}(v,ij)T(G_1) = \F_{H_1/ij}(u,ij)T(H_2) +\F_{H_2/ij}(v,ij)T(H_1).\]
\normalsize
Hence the number of spanning 2-forests separating $u$ and $v$ in $G$ is given by 
\begin{align*}\F_G(u,v) &= \F_H(u,v)-\F_{H_1}(u,i)\F_{H_2}(v,j)-\F_{H_1}(u,j)\F_{H_2}(v,i)\\
&+\F_{G_1}(u,i)\F_{G_2}(v,j)+\F_{G_1}(u,j)\F_{G_2}(v,i).\end{align*}
Recalling that $i=k$ and $j=k+1$, we have
\begin{align*}\F_G(u,v) &= \F_H(u,v)-\F_{H_1}(u,k)\F_{H_2}(v,k+1)-\F_{H_1}(u,k+1)\F_{H_2}(v,k)\\
&+\F_{G_1}(u,k)\F_{G_2}(v,k+1)+\F_{G_1}(u,k+1)\F_{G_2}(v,k).\end{align*}
We note that 
\begin{multline*}
\F_{G_1}(u,k+1) =\F_{H_1}(u,k+1), \F_{G_1}(u,k) =\F_{H_1}(u,k),\\ 
\F_{G_2}(v,k+1) =\F_{H_2}(v,k)\ {\rm and}\  \F_{G_2}(v,k) =\F_{H_2}(v,k+1)
\end{multline*}
Making these substitutions,
\begin{align*}\F_G(u,v) &= \F_H(u,v)-\F_{H_1}(u,k)\F_{H_2}(v,k+1)-\F_{H_1}(u,k+1)\F_{H_2}(v,k)\\
&+\F_{H_1}(u,k)\F_{H_2}(v,k)+\F_{H_1}(u,k+1)\F_{H_2}(v,k+1)\\
&=\F_H(u,v) + [\F_{H_1}(u,k+1) - \F_{H_1}(u,k)] [\F_{H_2}(v,k+1)-F_{H_2}(v,k)]
\end{align*}
Recalling that $H_1$ has $k+1$ vertices, we let $n=k+1$ in equation \eqref{eq:resdiststraightsum}.  Then applying it twice, once with $v=k+1$ and once with $v=k$ we obtain
\[
\F_{H_1}(u,k+1)-\F_{H_1}(u,k) = F_{k+1-u} F_{k-1+u} - F_{k-u} F_{k-2+u}.
\]
By Catalan's identity, \smallskip

$F_{k+1-u} F_{k-1+u}= F_k^2+(-1)^{k-u}F_{u-1}^2$ and $F_{k-u} F_{k-2+u}=F_{k-1}^2+(-1)^{k-u+1}F_{u-1}^2$.\medskip

\noindent Substituting in the previous equation and simplifying, we have 
\[
\F_{H_1}(u,k+1)-\F_{H_1}(u,k)=F_{k-2}F_{k+1}+2(-1)^{k-u} F_{u-1}^2.
\]

It remains to deal with the term $\F_{H_2}(v,k+1)-F_{H_2}(v,k)$.  In order to apply equation (\ref{eq:resdiststraightsum}) we must first adjust the vertex labels on $H_2$.
Let $R$ be the graph obtained from $H_2$ by labeling the vertices from the right beginning with $1$. 

\begin{figure}[h!]
\begin{center}
\begin{tikzpicture}[line cap=round,line join=round,>=triangle 45,x=1.0cm,y=1.0cm]
\begin{scope}[xshift = .45\textwidth]
\draw [line width=.8pt] (-2.,3.)-- (-1.5,3.8660254037844393);
\draw [line width=.8pt] (-1.5,3.8660254037844393)-- (-1.,3.);
\draw [line width=.8pt] (-1.,3.)-- (-2.,3.);
\draw [line width=.8pt] (-1.5,3.8660254037844393)-- (-0.5,3.866025403784439);
\draw [line width=.8pt] (-1.,3.)-- (-0.5,3.866025403784439);
\draw [line width=.8pt,dash pattern=on 2pt off 2pt] (-1.,3.)-- (0.,3.);
\draw [line width=.8pt,dash pattern=on 2pt off 2pt] (-0.5,3.866025403784439)-- (0.5,3.8660254037844375);
\draw [line width=.8pt] (0.5,3.8660254037844375)-- (0.,3.);
\draw [line width=.8pt] (0.,3.)-- (1.,3.);
\draw [line width=.8pt] (1.,3.)-- (0.5,3.8660254037844375);
\draw [line width=.8pt] (0.5,3.8660254037844375)-- (1.5,3.8660254037844357);
\draw [line width=.8pt] (1.5,3.8660254037844357)-- (1.,3.);
\draw [line width=.8pt] (-2.5,3.8660254037844397)-- (-1.5,3.8660254037844393);
\draw [line width=.8pt] (-1.5,3.8660254037844393)-- (-2.,3.);
\draw [line width=.8pt] (-2.,3.)-- (-1.,3.);
\begin{scriptsize}
\draw [fill=black] (-2.,3.) circle (1.6pt);
\draw[color=black] (-2.34,2.77) node {$n-k$};
\draw [fill=black] (-1.,3.) circle (1.6pt);
\draw [fill=black] (-1.5,3.8660254037844393) circle (1.6pt);
\draw [fill=black] (-0.5,3.866025403784439) circle (1.6pt);
\draw [fill=black] (0.,3.) circle (1.6pt);
\draw [fill=black] (0.5,3.8660254037844375) circle (1.6pt);
\draw [fill=black] (1.,3.) circle (1.6pt);
\draw[color=black] (1.08,2.81) node {$2$};
\draw [fill=black] (1.5,3.8660254037844357) circle (1.6pt);
\draw[color=black] (1.7,4.03) node {$1$};
\draw [fill=black] (-1.,3.) circle (1.6pt);
\draw [fill=black] (-2.5,3.8660254037844397) circle (1.6pt);
\draw[color=black] (-2.36,4.13) node {$n-k+1$};
\end{scriptsize}
\end{scope}
\end{tikzpicture}
\end{center}
\caption{The graph $R$ obtained by reordering the vertices of $H_2$.}\label{fig:BH2}
\end{figure}
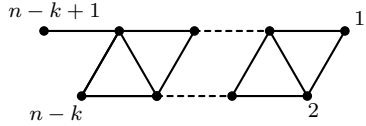

By symmetry $\F_{H_2}(v,k+1)=\F_R(n-v+1,n-k)$ and $\F_{H_2}(v,k)=\F_R(n-v+1,n-k+1)$.  Let $S$ be the graph obtained from $R$ by deleting vertex $n-k+1$.  Because $n-k-1$ is a cut vertex of $R$,
\[ \F_R(n-v+1,n-k)=\F_S(n-v+1,n-k)\]
and
\[ \F_R(n-v+1,n-k+1) = \F_S(n-v+1,n-k-1) + T(H_{n-k}).\]
Then 
\begin{multline*}
 \F_{H_2}(v,k+1)-F_{H_2}(v,k)= \F_R(n-v+1,n-k) - \F_R(n-v+1,n-k+1)\\
 =\F_S(n-v+1,n-k) - \F_S(n-v+1,n-k-1) - T(H_{n-k}).
 \end{multline*}
Applying equation (\ref{eq:resdiststraightsum}) to the $\F_S$ terms and recalling that $T(H_{n-k})=F_{2n-2k-2}$ yields
\begin{multline*} 
\F_{H_2}(v,k+1)-\F_{H_2}(v,k)=F_{v-k-1}F_{2n-v-k-1}-F_{v-k-2}F_{2n-v-k-2}-F_{2n-2k-2}\\
=F_{n-k-1}^2+(-1)^{v-k}F_{n-v}^2-F_{n-k-2}^2-(-1)^{v-k-1}F_{n-v}^2-F_{2n-2k-2}\ \ {\rm (Catalan's\ identity)}\\
=F_{n-k-3}F_{n-k}-F_{2n-2k-2}-2(-1)^{v-k-1}F_{n-v}^2\\
=-F_{n-k-2}F_{n-k+1}-2(-1)^{v-k-1}F_{n-v}^2.
\end{multline*}
This completes the proof.
\end{proof}\bigskip

\begin{corollary}Let $G_n$ be a bent linear 2-tree on $n$ vertices with a single bend at vertex $k$.  Then for any two vertices $u$ and $v$ with $u \le k$ and $v>k+1$, 
\[
\F_{G_n}(u,v) < \F_{H_n}(u,v).
\]
\end{corollary}

\begin{corollary}\label{cor:27}
 Let $G_n$ be a bent linear 2-tree with $n$ vertices and a single bend at vertex $k$.  Then 
  \[ \F_{G_n}(1,n)=\F_{H_n}(1,n)- F_{k-2}F_{k+1}F_{n-k-2}F_{n-k+1}.\]
  and
  \[ r_{G_n}(1,n)= r_{H_n}(1,n)   -\frac{F_{k-2}F_{k+1}F_{n-k-2}F_{n-k+1}}{F_{2n-2}}.\]
  \end{corollary}
  \noindent The second terms tell exactly how much the number of separating 2-forests and resistance distance are diminished by a bend at $k$.

  \begin{proof}
Using Theorem~\ref{thm:bigone} and setting $u=1$ and $v=n$ yields the first equality and the second follows immediately.
\end{proof}

  \begin{remark}
  The resistance distance $r_{H_n}(1,n)$ is known from ~\cite{bef}.  Substituting gives 
  \[ r_{G_n}(1,n) = \frac{n-1}{5}+\frac{4F_{n-1}}{5L_{n-1}}-\frac{F_{k-2}F_{k+1}F_{n-k-2}F_{n-k+1}}{F_{2n-2}},\]
\noindent where $F_p$ is the $p$th Fibonacci number and $L_q$ is the $q$th Lucas number. 

	\end{remark}

\section{Conclusion}
In this paper we have given a non-trivial generalization of well-known and elementary formulae for calculating the number of spanning trees of a graph $G$ and the number of separating 2-forests  separating 2 vertices in the case that $G$ contains a cut-vertex, to the case in which $G$ contains a 2-separator.  We have applied these formulae to the Sierpinski triangle and to a family of linear 2-trees with a single bend.  For the Sierpinski triangle we observed that for any of the degree 2 vertices $a$ and $b$, the resistance distance $r_n(a,b) \to \infty$ as $n\to \infty$.  For the bent linear 2-tree we determined the resistance distance (or number of spanning 2-forests) between any pair of vertices.  We found that if two vertices are strictly on one side of a 2-separator associated with the bend, then the resistance distances in the bent 2-tree match the corresponding resistance distances in the straight linear 2-tree on the same number of vertices, while if the two vertices are on opposite sides of this 2-separator,  all the resistance distances in the bent 2-tree are less than the corresponding resistance distances in the straight linear 2-tree.   Nevertheless, it is straightforward to check that $r_{G_n}(1,n) \to \infty$ as $n \to \infty$ if $G_n$ is any member of the family of linear 2-trees with a single bend as is the case with the straight linear 2-tree~\cite{bef}.  

These two examples illustrate the power of these 2-separation formulae, and we are confident that they can be applied to many other families of graphs in which standard circuit rules (for resistance distance) or combinatorial arguments (to count separating spanning 2-forests) are difficult, tedious, or even impossible to apply. Finding analogous formulae for 3-connected graphs is an interesting but difficult, open question.

\bibliography{references}{}
\bibliographystyle{plain} 

\end{document}